\documentclass[11pt]{amsart}

\usepackage{amsmath,amssymb,latexsym,soul,cite,mathrsfs}

\usepackage{color,enumitem,graphicx}
\usepackage[colorlinks=true,urlcolor=blue,
citecolor=red,linkcolor=blue,linktocpage,pdfpagelabels,
bookmarksnumbered,bookmarksopen]{hyperref}
\usepackage[english]{babel}

\usepackage[left=2.9cm,right=2.9cm,top=2.8cm,bottom=2.8cm]{geometry}
\usepackage[hyperpageref]{backref}

\usepackage[colorinlistoftodos]{todonotes}
\makeatletter
\providecommand\@dotsep{5}
\def\listtodoname{List of Todos}
\def\listoftodos{\@starttoc{tdo}\listtodoname}
\makeatother

\numberwithin{equation}{section}
\newtheorem{theorem}{Theorem}[section]
\newtheorem{proposition}[theorem]{Proposition}
\newtheorem{lemma}[theorem]{Lemma}
\newtheorem{corollary}[theorem]{Corollary}

\newtheorem{step}{Step}

\newcommand\R{\mathbb R}
\newcommand\N{\mathbb N}

\begin{document}

\title[Fractional elliptic problems on exterior domains]
{Existence of solutions for a class of fractional elliptic problems on exterior domains}

\author{Claudianor O. Alves}
\author{Giovanni Molica Bisci}
\author{C\'esar E. Torres Ledesma}
%\author{Gaetano Siciliano}

\address[Claudianor O. Alves]{\newline\indent Unidade Acad\^emica de Matem\'atica
\newline\indent
Universidade Federal de Campina Grande,
\newline\indent
58429-970, Campina Grande - PB - Brazil}
\email{\href{mailto:coalves@mat.ufcg.edu.br}{coalves@mat.ufcg.edu.br}}

\address[Giovanni Molica Bisci]{\newline\indent Dipartimento P.A.U.
	\newline\indent
Universit\'a degli Studi Mediterranea di Reggio Calabria,
	\newline\indent
Salita Melissari - Feo di Vito, 89100 Reggio Calabria, Italy}
\email{\href{mailto:gmolica@unirc.it}{gmolica@unirc.it}}

\address[C\'esar E. Torres Ledesma]
{\newline\indent Departamento de Matem\'aticas
\newline\indent
Universidad Nacional de Trujillo
\newline\indent
Av. Juan Pablo II s/n. Trujillo-Per\'u}
\email{\href{ctl\_576@yahoo.es}{ctl\_576@yahoo.es}}

\pretolerance10000

%\begin{document}

\begin{abstract}
\noindent
This work concerns with the existence of solutions for the following class of nonlocal elliptic problems
\begin{equation*}\label{00}
\left\{
\begin{array}{l}
(-\Delta)^{s}u + u = |u|^{p-2}u\;\;\mbox{in $\Omega$},\\
u \geq 0 \quad \mbox{in} \quad \Omega \quad \mbox{and} \quad u \not\equiv 0, \\
u=0 \quad \mathbb{R}^N \setminus \Omega,
\end{array}
\right.
\end{equation*}
involving the fractional Laplacian operator $(-\Delta)^{s}$,
where $s\in (0,1)$, $N> 2s$, $\Omega \subset \R^N$ is an exterior domain with (non-empty) smooth boundary $\partial \Omega$ and $p\in (2, 2_{s}^{*})$. The main technical approach is based on variational and topological methods. The variational analysis that we perform in this paper dealing with exterior domains is quite general and may be suitable for other goals too.
\end{abstract}

\thanks{C.O. Alves was partially supported by CNPq/Brazil 304804/2017-7 and C.E. Torres Ledesma was partially supported by INC Matem\'atica  88887.136371/2017. }
\subjclass[2010]{Primary 35J60; Secondary 35C20, 35B33, 49J45.}
\keywords{}

\maketitle

%------------------------------------------------------------------------------
\section{Introduction}
%------------------------------------------------------------------------------

In this paper we study the existence of a solution for the following fractional elliptic problem
$$
\left\{
\begin{array}{l}
(-\Delta)^{s}u + u = |u|^{p-2}u\;\;\mbox{in $\Omega$}\\
u \geq 0 \quad \mbox{in} \quad \Omega \quad \mbox{and} \quad u \not\equiv 0, \\
u=0 \quad \mathbb{R}^N \setminus \Omega,
\end{array}
\right.
\leqno{(P)}
$$
where $s\in (0,1)$, $N> 2s$, $\Omega \subset \R^N$ is an exterior domain, i.e. an unbounded domain with smooth boundary $\partial \Omega \neq \emptyset$ such that $\R^N \setminus \Omega$ is bounded, $p\in (2, 2_{s}^{*})$, where $2_s^* = \frac{2N}{N-2s}$ is the fractional critical Sobolev exponent and $(-\Delta)^s$ is the classical fractional Laplace operator.

When $s\nearrow 1^-$, problem ($P$) reduces to the following elliptic problem
\begin{equation}\label{I01}
\left\{
\begin{array}{l}
-\Delta u + u = |u|^{p-2}u\;\;\mbox{in $\Omega$}\\
u= 0 \quad \mbox{on} \quad \partial \Omega,
\end{array}
\right.
\end{equation}
with $p\in (2, 2^*)$ and $2^* = \frac{2N}{N-2}$. This problem was studied by Benci and Cerami in \cite{VBGC}, and they proved that (\ref{I01}) does not have a ground state solution, which becomes a difficulty in dealing with the problem. The authors analyzed the behavior of Palais-Smale sequences and showed a precise estimate of the energy levels where the Palais-Smale condition fails, which made possible to show that the problem (\ref{I01}) has at least one positive solution, for $\R^N \setminus \Omega$ small enough. A key point in the approach explored in \cite{VBGC} is the existence and uniqueness, up to a translation, of a positive solution  $\Psi$ of the limit problem associated with (\ref{I01}) given by
\begin{equation}\label{DL}
\left\{
\begin{array}{l}
-\Delta u + u = |u|^{p-2}u \quad \mbox{in} \,\,\, \R^N \\
u \in H^{1}(\mathbb{R}^N).
\end{array}
\right.
\end{equation}
Moreover, the fact that $\Psi$ is  radially symmetric about the origin, monotonically decreasing in $|x|$, and that has an exponential decay apply important role in some estimates.
For related problems involving exterior domain we cite Alves and Freitas  \cite{CALF}, Bahri and Lions \cite{bah}, Cerami and Passaseo \cite{BP1}, Citti \cite{citti},  Clapp and Salazar  \cite{clapp}, Coffman and Marcus \cite{coffman}, Li and Zheng \cite{li}, Maia and Pellacci \cite{maia}, and their references.

Recently, the case $s \in (0,1)$ has received a special attention, because involves the fractional Laplacian operator $(-\Delta)^{s}$, which arises in a quite natural way in many different contexts, such as, among the others, the thin obstacle problem, optimization, finance, phase transitions, stratified materials, anomalous diffusion, crystal dislocation, soft thin films, semipermeable membranes, flame propagation, conservation laws, ultra-relativistic limits of quantum mechanics, quasi-geostrophic flows, multiple scattering, minimal surfaces, materials science and water waves, for more detail see \cite{Bucurb, EDNGPEV, Dipierrob, Molicab, CP}.

The reader can find in the literature very interesting papers whose the existence of solution has been established for problems like
\begin{equation}\label{I02*}
(-\Delta)^{s}u + V(x)u = f(x,u)\;\;\mbox{in}\;\; \mathbb{R}^{N},
\end{equation}
where $V$ and $f$ verify suitable conditions, see for example Alves and Miyagaki \cite{AlvesMIyagaki1}, Alves, de Lima and N\'obrega \cite{Alves-delima-Nobrega}, Autuori and Pucci \cite{Autuori}, Felmer, Quaas and Tan  \cite{PFAQJT},  Cheng \cite{MC}, Secchi \cite{SS-1},  D\'avila, del Pino and Wei \cite{Davila}, Dipierro, Patalucci and Valdinoci \cite{SDGPEV}, Fall, Mahmoudi and Valdinoci \cite{Moustapha}, Molica Bisci and R\u{a}dulescu \cite{MOLRAD}, Servadei and Valdinoci \cite{RSEV,RSEV1}, Shang and Zhang \cite{ShangZhang, Shang}, Caponi and Pucci \cite{Caponi}, Fiscella, Pucci and Saldi \cite{Fiscella} and references therein. Here, we would like point out that in Frank and Lenzmann \cite{RFEL} and Frank, Lenzmann and Silvestre \cite{RFELLS} the existence and uniqueness (up to symmetries) of positive ground state solution $Q$ was proved for the problem
\begin{equation}\label{infinito}
(-\Delta)^{s}u + u = |u|^{p-2}u\;\;\mbox{in}\;\; \mathbb{R}^{N},
\end{equation}
for every $p\in (2, 2_{s}^{*})$. Moreover, $Q$ is  radially symmetric about the origin and monotonically decreasing in $|x|$. On the contrary of the classical elliptic case, for $s \in (0,1)$ any information is available about the exponential decay of $Q$.

Since we did not find in the literature any paper dealing with the existence of non negative solutions for problem ($P$) in exterior domains, motivated by the ideas found in Benci and Cerami \cite{VBGC},  we intend in the present paper to prove that ($P$) has a nontrivial weak solution. As above mentioned, in \cite{VBGC}, Benci and Cerami used the fact that positive ground state solution $\Psi$ of (\ref{DL}) has an exponential decaying to prove some estimates, however for fractional Laplacian this type of behavior was not established yet, which brings some technical difficulty to prove the existence of solution for ($P$). However, we were able to proof that the exponential decay infinity is not necessary to establish the existence of a non negative solution for $(P)$.

Our main result is the following:

\begin{theorem}\label{teo 1}
	There exists $\rho^* >0$ such that if $\R^N \setminus \Omega \subset B_{\rho}(0)$ and $\rho < \rho^*$, then problem  $(P)$ has at least one non negative solution.
\end{theorem}

This work is organized as follows. In Section 2, we introduce some preliminary results that will be used in the paper. In Section 3, we show an important compactness result for energy functional, which is a key point in our arguments. In Section 4, we prove some estimates that will be used in Section 5 to prove Theorem \ref{teo 1}.

\section{Preliminary results}

For $s\in (0,1)$ and $N>2s$, the fractional Sobolev space of order $s$ on $\mathbb{R}^N$ is defined by
$$
H^{s}(\mathbb{R}^N) = \left\{ u\in L^{2}(\mathbb{R}^N):\;\;\int_{\mathbb{R}^N}\int_{\mathbb{R}^N} \frac{|u(x)-u(z)|^2}{|x-z|^{N+2s}}dzdx<\infty \right\},
$$
endowed with the norm
$$
\|u\|_{s} = \left( \int_{\mathbb{R}^N} |u(x)|^2 dx + \int_{\mathbb{R}^N} \int_{\mathbb{R}^N} \frac{|u(x) - u(z)|^2}{|x-z|^{N+2s}}dzdx\right)^{1/2}.
$$

\indent We recall the fractional version of the Sobolev embeddings (see \cite{PFAQJT}).
\begin{theorem}\label{sobolev}
Let $s\in (0,1)$, then there exists a positive constant $C=C(N,s)>0$ such that
\begin{equation}\label{P01}
\|u\|_{L^{2_{s}^{*}}(\R^N)}^{2} \leq C\int_{\R^N}\int_{\R^N}\frac{|u(x) - u(y)|^2}{|x-y|^{N+2s}}dy dx
\end{equation}
and then $H^s(\R^N)\hookrightarrow L^q(\R^N)$ is continuous for all $q\in [2, 2_{s}^{*}]$. {{Moreover, if $\Theta \subset \mathbb{R}^N$ is a bounded domain, we have that the embedding $H^s(\R^N)\hookrightarrow L^q(\Theta)$ is compact for any $q\in [2, 2_{s}^{*})$.}}
\end{theorem}
{Hereafter, we denote by $X_{0}^s \subset H^s(\R^N)$ the subspace defined by}
$$
X_{0}^s = \Bigl\{ u\in H^s(\R^N):\;\;u = 0\;\;\mbox{a.e. in}\;\R^N\setminus \Omega  \Bigr\},
$$
endowed with the norm $\|\cdot\|_s$. Moreover we introduce the following norm
\begin{equation}\label{P02}
\|u\|= \left(\int_{\Omega} |u(x)|^2dx + \iint_{\mathcal{Q}} \frac{|u(x)-u(z)|^2}{|x-z|^{N+2s}}dzdx\right)^{\frac{1}{2}},
\end{equation}
where $\mathcal{Q} = \R^{2N}\setminus (\Omega^c\times \Omega^c)$. {We point out that $\|u\|_s=\|u\|$ for any  $u\in X_{0}^{s}$. Since $\partial \Omega$ is bounded and smooth, by \cite[Theorem 2.6]{Molicab}, we have the following result.}
\begin{theorem}\label{density}
 The space $C_{0}^{\infty}(\Omega)$ is dense in  $(X_{0}^{s},\|\cdot\|)$.
\end{theorem}
\par
\smallskip
\noindent
In what follows, we denote by $H^{s}(\Omega)$ the usual fractional Sobolev space endowed with the norm
$$
\|u\|_{H^{s}} = \left(\int_{\Omega}|u(x)|^2dx + \int_{\Omega}\int_{\Omega} \frac{|u(x) - u(z)|^2}{|x-z|^{N+2s}}dzdx\right)^{\frac{1}{2}}.
$$
Related to these fractional spaces, we have the following properties
{{\begin{proposition}\label{Pprop01}
The following assertions hold true:
\begin{itemize}
\item[$(i)$] If $v\in X_{0}^s$, we have that $H^{s}(\Omega)$ and
$$
\|v\|_{H^{s}} \leq \|v\|_{s} = \|v\|.
$$
\item[$(ii)$] Let $\Theta$ an open set with continuous boundary. Then, there exists a positive constant $\mathfrak{C} = \mathfrak{C}(N,s)$, such that
$$
\|u\|_{L^{2_{s}^{*}}(\Theta)}^{2} = \|u\|_{L^{2_{s}^{*}}(\R^N)}^{2}\leq \mathfrak{C} \iint_{\mathbb{R}^{2N}} \frac{|u(x) - u(z)|^2}{|x-z|^{N+2s}}dzdx, \quad ( \mbox{see \rm \cite[Theorem 6.5]{EDNGPEV}}\,)
$$
for every $u\in X_0^s$.
\end{itemize}
\end{proposition}}}

%\textcolor{red}{We cannot cite [27] here as reference for Proposition 2.3. It is necessary to add some comments about the validity of this proposition.}\par

The following lemma is a fractional version of the concentration compactness principle due to Lions, whose the proof can be seen in \cite{PFAQJT}.
\begin{lemma}\label{CClem}
Let $N\geq 2$ and $q\in [2, 2_{s}^{*})$. Assume that $\{u_{n}\} \subset H^{s}(\mathbb{R}^{N})$ is a bounded sequence satisfying
\begin{equation}\label{P03}
\lim_{n\to \infty} \sup_{y\in \mathbb{R}^{N}}\int_{B(y,R)}|u_{n}(x)|^{q}dx = 0,
\end{equation}
for some $R>0$. Then $u_{n} \to 0$ in $L^{p}(\mathbb{R}^{N})$ for $2 < p < 2_{s}^{*}$.
\end{lemma}

From now on, $M_\infty$ designates the following constant
\begin{equation}\label{P04}
M_\infty := \inf \left\{\|u\|_{s}^{2}:\;u\in H^s(\R^N)\;\;\mbox{and}\;\;\int_{\R^N}|u(x)|^pdx = 1  \right\},
\end{equation}
%where
%$$
%\|u\|_s= \left(\int_{\R^N}\int_{\R^N} \frac{|u(x) - u(y)|^2}{|x-y|^{N+2s}}dy dx + \int_{\R^N}|u(x)|^2dx \right)^{1/2}.
%$$
which is positive  by Theorem \ref{sobolev}. Furthermore, for any $v\in H^s(\R^N)$ and $z\in \R^N$, we set the function
$$
v^{z}(x) := v(x+z).
$$
Then, by doing the change of variable $\tilde{x} = x+z$ and $\tilde{y} = y+z$, it is easily seen that
\begin{equation}\label{P07}
\|v^z\|_{s}^{2} = \|v\|_{s}^{2}\;\;\;\mbox{as well as} \,\,\, \|v^z\|_{L^p(\mathbb{R}^N)} = \|v\|_{L^p(\mathbb{R}^N)}.
\end{equation}

Arguing as in  \cite{VBGC} the following result holds true.

\begin{theorem}\label{tm01}
Let $\{u_n\} \subset H^s(\R^N)$ be a minimizing sequence such that
$$
\|u_n\|_{L^p(\mathbb{R}^N)} = 1\;\;\mbox{and}\;\;\|u_n\|_{s}^{2} \to M_\infty \;\;\mbox{as}\;\;n\to +\infty.
$$
Then, there is a sequence $\{y_n\}\subset \R^N$ such that $\{u_{n}^{y_n}\}$ has a convergent subsequence, and so, $M_\infty$ is attained.
\end{theorem}

As a byproduct of the above result the next corollary is obtained.

\begin{corollary}
\label{C1}There is $v \in H^{s}(\mathbb{R}^N)$ such that $\|v\|_s=M_\infty$ and $\|v\|_{L^{p}(\mathbb{R}^N)}=1$.	
\end{corollary}

\section{A compactness result for energy functional}
 In this section, we establish the existence of non negative solution for problem ($P$).
 %following problem
 %$$
%\left\{
%\begin{aligned}
%(-\Delta)^{s}u + \lambda u &= |u|^{p-2}u\;\;\mbox{in}\;\;\Omega\\
%u&\in X_{0}^{s},
%\end{aligned}
%\right.
%\leqno{(P_2)}
%$$
%where $\Omega \subset \R^N$ be a exterior domains, that is an unbounded domain with smooth boundary $\partial \Omega \neq \emptyset$ such that $\R^N \setminus \Omega$ is bounded, $\lambda \in \R^+$, $s \in (0,1)$, $2s < N$ and $q\in (2, 2_{s}^{*})$.
Through this section we fix on $X_{0}^s$ the norm
$$
\|u\| := \left(\iint_{\mathcal{Q}} \frac{|u(x)-u(y)|^2}{|x-y|^{n+2s}}dy dx + \int_{\Omega}|u|^2dx\right)^{1/2}.
$$
and denote by $M>0$ the number
\begin{equation}\label{E02}
M := \inf \left\{\|u\|^2:\;\;u\in X_{0}^s,\;\;\int_{\Omega}|u(x)|^pdx = 1 \right\}.
\end{equation}

\begin{theorem}\label{Etm01}
The equality $M_\infty=M$ holds true. {Hence, there is no $u \in X_0^{s}$ such that $\|u\|^{2}=M$ and $\|u\|_{L^p(\mathbb{R}^N)}=1$, and so, the minimization problem \eqref{E02} does not have solution.} %has not ground state solution
\end{theorem}
\begin{proof}
By Proposition \ref{Pprop01} - part (i) it follows that
\begin{equation}\label{E03}
M_\infty \leq M.
\end{equation}
Let $\varphi$ be a minimizer of (\ref{P04}), that is
\begin{equation}\label{E04}
\varphi \in H^s (\R^N),\;\;\int_{\R^N}|\varphi|^pdx = 1\;\;\mbox{and}\;\;M_\infty = \|\varphi\|_{s}^{2}.
\end{equation}
In addition, let
$\{y_n\}\subset \Omega$ be a sequence such that $|y_n|\to +\infty$ as $n\to +\infty$, and $\rho$ be the smallest positive number satisfying
$$
\R^N \setminus \Omega \subset B(0,\rho) = \{x\in \R^N:|x|< \rho\}.
$$
Furthermore, let us fix $\zeta \in C^\infty(\R^N, [0,1])$ defined by
$$
\zeta (x) := \xi \left(\frac{|x|}{\rho} \right),
$$
where $\xi : [0,+\infty) \to [0,1]$ is a non-decreasing function such that
$$
\xi (t) := 0, \;\;\forall t\leq 1\;\;\mbox{and}\;\;\xi (t) := 1,\;\;\forall t\geq 2.
$$
With the above notations, define
$$
\phi_n(x) := c_n \zeta(x)\varphi (x-y_n),
$$
where $c_n$ is the normalization constant given by
$$
c_n := \|\zeta(x)\varphi (x-y_n)\|_{L^p(\mathbb{R}^N)}^{-1}.
$$
%Now we are going to show that
We claim that
\begin{equation}\label{E05}
\|\phi_n\|^{2}_s \to M_\infty \;\;\mbox{as}\;\;n\to +\infty.
\end{equation}
Indeed, after the change of variable $z = x-y_n$, we get
$$
\begin{aligned}
\|\zeta (x) \varphi (x-y_n)  - \varphi (x-y_n)\|_{L^p(\mathbb{R}^N)} %& = \left( \int_{\R^N} |\zeta (x)\varphi (x-y_n) - \varphi (x-y_n)|^pdx\right)^{\frac{1}{p}}\\
%&= \left(\int_{\R^N} |\zeta (z+y_n)\varphi (z) - \varphi (z)|^pdz \right)^{\frac{1}{p}} \\
&= \left(\int_{\R^N}f_n(z)dz \right)^{\frac{1}{p}},
\end{aligned}
$$
where $f_n(z) = |(\zeta (x+y_n)-1)\varphi(z)|^p$. Since $|y_n| \to +\infty$ as $n\to +\infty$, it follows that
$$
f_n(z) \to 0\;\;\mbox{a.e. in}\;\;\R^N.
$$
Now, taking into account that
$$
\begin{aligned}
|f_n(z)| & = |\zeta (z+y_n) - 1|^p|\varphi (z)|^p \leq (|\zeta (z+y_n)| + 1)^p \leq 2^p|\varphi (z)|^p \in L^1(\R^N),
\end{aligned}
$$
the Lebesgue's theorem yields
$$
\int_{\R^N} f_n(z)dz \to 0\;\;\mbox{as}\;\;n\to +\infty.
$$
Therefore
$$
\zeta (\cdot +y_n) \varphi \to \varphi\;\;\mbox{in}\;\;L^p(\R^N).
$$
A similar argument ensures that
$$
c_n \to 1\;\;\mbox{as}\;\,n\to +\infty
$$
and
\begin{equation}\label{E06}
\int_{\R^N} [\zeta (x)\varphi (x-y_n) - \varphi (x-y_n)]^2dx = o_n(1).
\end{equation}
Now, we claim that
\begin{equation}\label{E07}
\int_{\R^N}\int_{\R^N}\frac{|(\zeta (x)-1)\varphi (x - y_n) - (\zeta (y) - 1)\varphi (y-y_n)|^2}{|x-y|^{N+2s}}dy dx = o_n(1).
\end{equation}
Indeed, let
$$
\Phi_u(x,y) := \frac{u(x)-u(y)}{|x-y|^{\frac{N}{2}+s}}.
$$
Then, after the change of variables $\tilde{x} = x-y_n$ and $\tilde{y} = y-y_n$, one has
$$
\begin{aligned}
&\int_{\R^N}\int_{\R^N}\frac{|(\zeta (x)-1)\varphi (x - y_n) - (\zeta (y) - 1)\varphi (y-y_n)|^2}{|x-y|^{N+2s}}dy dx
%&= \int_{\R^N}\int_{\R^N}\frac{|(\zeta (x + y_n)-1)\varphi (x) - (\zeta (y+y_n) - 1)\varphi (y)|^2}{|x-y|^{N+2s}}dy dx\\
= \int_{\R^N}\int_{\R^N}|\Phi_{n}(x,y)|^2dy dx
\end{aligned}
$$
where
$$
\Phi_{n}(x,y) := \frac{(\zeta (x+y_n)-1)\varphi (x) - (\zeta (y + y_n) - 1)\varphi (y)}{|x-y|^{\frac{N}{2}+s}}.
$$
Recalling that $\displaystyle\lim_{n\to +\infty} |y_n| = +\infty$, we also have
\begin{equation}\label{E08}
\Phi_{n}(x,y)\to 0\;\;\mbox{a.e. in}\;\;\R^N \times \R^N.
\end{equation}
On the other hand, a direct application of the mean value theorem yields
\begin{equation}\label{E09}
\begin{aligned}
|\Phi_{n}(x,y)| &\leq |1-\zeta (x+y_n)||\Phi_\varphi (x,y)| + |\varphi (y)||\Phi_{1-\zeta}(x+y_n, y+y_n)|\\
&\leq  |\Phi_\varphi (x,y)| + \frac{C|\varphi (y)|}{|x-y|^{\frac{N}{2}+s -1}}\chi_{B(y,1)}(x) + \frac{2|\varphi(y)|}{|x-y|^{\frac{N}{2}+s}}\chi_{B^{c}(y,1)}(x),
\end{aligned}
\end{equation}
for almost every $(x,y)\in \R^N \times \R^N$. Now, it easily seen that the right hand side in (\ref{E09}) is $L^2$-integrable. Thus, the Lebesgue's theorem immediately yields relation (\ref{E07}). Therefore, by (\ref{E06}) and (\ref{E07}), it follows that
\begin{equation}\label{E10}
\|\zeta(\cdot)\varphi (\cdot-y_n) - \varphi(\cdot-y_n)\|_s^2 \to 0\;\;\mbox{as}\;\;n\to +\infty.
\end{equation}
Now, since $\varphi$ is a minimizer of (\ref{E02}), one has
$$
\|\zeta(\cdot)\varphi(\cdot - y_n)\|_{s}^{2} = \|\varphi(\cdot-y_n)\|_{s}^{2} + o_n(1) = \|\varphi\|_{s}^{2} + o_n(1) = M_\infty + o_n(1).
$$
Similar arguments ensures that
\begin{equation}\label{E11}
\begin{aligned}
&\|\phi_n\|_{s}^{2} = \|\phi_n\|^2 = M_\infty + o_n(1)
\end{aligned}
\end{equation}
and
\begin{equation}\label{E12}
\|\phi_n\|_{L^p(\mathbb{R}^N)} = \|c_n \zeta(\cdot)\varphi (\cdot- y_n)\|_{L^p(\mathbb{R}^N)} = c_n\|\zeta(\cdot)\varphi (\cdot - y_n)\|_{L^p(\mathbb{R}^N)} = 1.
\end{equation}
Thereby, by definition of $M$ and (\ref{E11}),
\begin{equation}\label{E13}
M \leq M_\infty.
\end{equation}
By using (\ref{E03}) and (\ref{E13}),
\begin{equation}\label{E14}
M = M_\infty,
\end{equation}
which proves the first part of the main result.\par
{Now, suppose by contradiction that there is $v_0\in X_0^s$ satisfying
$$\|v_0\| = M \;\;\mbox{and} \;\;\|v_0\|_{L^p(\Omega)} = 1.$$
%By last equality we have
%$$
%\|v_0\| = M_\infty \;\;\mbox{and}\;\;\|v_0\|_{L^p(\mathbb{R}^N)} = 1.
%$$
Without loss of generality, we can assume that $v_0 \geq 0$ in $\Omega$. Note that by (\ref{E14}), since $v_0 \in H^{s}(\mathbb{R}^N)$ and $\|v_0\|=\|v_0\|_s$, it follows that $v_0$ is a minimizer for (\ref{P04}), and so, a solution of problem
 $$
\left\{
\begin{aligned}
(-\Delta)^{s}u +  u &= M_\infty u^{p-1}\;\;\mbox{in}\;\;\R^N\\
u&\in H^s (\R^N).
\end{aligned}
\right.
\leqno{(P_2)}
$$
Therefore, by the maximum principle we get that $v_0>0$ in $\R^N$, which is impossible, because $v_0 = 0$ in $\R^N \setminus \Omega$. This completes the proof.}
\end{proof}

\subsection{A compactness lemma}

In this section we prove a compactness result involving the energy functional $I:X_0^s \to \mathbb{R}$ associated to the main problem ($P$) and given by
$$
I(u) = \frac{1}{2}\left( \iint_{\mathcal{Q}}\frac{|u(x) - u(y)|^2}{|x-y|^{N+2s}}dy dx + \int_{\Omega} |u|^2dx\right) - \frac{1}{p}\int_{\Omega}|u|^pdx.
$$
Here and subsequently, we consider the problem
$$
\left\{
\begin{aligned}
(-\Delta)^{s}u +  u &=  |u|^{p-2}u\;\;\mbox{in}\;\;\R^N\\
u&\in H^s (\R^N),
\end{aligned}
\right.
\leqno{(P_\infty)}
$$
whose the energy functional $I_\infty: H^s (\R^N) \to \R$ is given by
$$
I_\infty(u) = \frac{1}{2}\left(\int_{\R^N}\int_{\R^N}\frac{|u(x) - u(y)|^2}{|x-y|^{N+2s}}dy dx + \int_{\R^N}|u|^2dx \right) - \frac{1}{p}\int_{\R^N}|u|^pdx.
$$
With the above notations we are abe to prove the following compactness result.
\begin{lemma}\label{Elm01}
Let $\{u_n\} \subset X_0^s$ be a sequence such that
\begin{equation}\label{E17}
I(u_n)\to c\;\;\mbox{and}\;\;I'(u_n)\to 0\;\;\mbox{as}\;\;n\to \infty.
\end{equation}
%Suppose that $u_n \nrightarrow \textcolor{red}{u^0}$ in $X_0^s$.
Then, up to a subsequence, {there exist a weak solution $u_0 \in X_{0}^{s}$ of $(P)$,  a number $k\in \N$, $k$ sequences $\{y_{n}^{j}\} \subset \R^N$ and $k$ functions $\{u_{n}^{j}\} \subset H^s (\R^N)$, $1\leq j \leq k$ such that}
$$
\begin{aligned}
&|y_{n}^{j}|\to +\infty \;\;\mbox{for}\;\;1\leq j \leq k,\\
&u_{n}^{0} = u_n \rightharpoonup u^0\;\;\mbox{in}\;\;X_0^s,\\
&u_{n}^{j} \rightharpoonup u^j\;\;H^s (\R^N) \;\;\mbox{for}\;\;1\leq j \leq k,
\end{aligned}
$$
where $u^j$ are nontrivial weak solution of $(P_\infty)$, for every $1\leq j \leq k$. Furthermore,
$$
\|u_n\|^2 \to \|u_0\|^2 + \sum_{j=1}^{k}\|u^j\|_{s}^{2}
$$
and
$$
I(u_n) \to I(u^0) + \sum_{j=1}^{k} I_\infty (u^j).
$$
\end{lemma}
\begin{proof} We divide the proof of this lemma into several steps.
\begin{step}\label{Eclaim1}
The sequence $\{u_n\}$ is bounded in $X_0^s$.
\end{step}
\begin{proof}
By using the definition of $I$, we notice that
\begin{equation}\label{E18}
I(u_n) - \frac{1}{p}I'(u_n)u_n = \left(\frac{1}{2} - \frac{1}{p}\right)\|u_n\|^2.
\end{equation}
Then, by (\ref{E17}) and (\ref{E18}), one has
\begin{equation*}\label{E20}
c_1 + \frac{c_2}{p}\|u_n\| \geq I(u_n) - \frac{1}{p}I'(u_n)u_n \geq \left(\frac{1}{2} - \frac{1}{p} \right)\|u_n\|^2,
\end{equation*}
The above inequality gives the boundedness of the sequence $\{u_n\}$ in $X_0^s$.
\end{proof}

Thanks to the reflexivity of $X_0^s$, up to a subsequence, by Step 1, there exists $u_0\in X_0^s$ such that
\begin{equation}\label{E22}
u_n \rightharpoonup  u^0\;\;\mbox{in}\;\;X_0^s\quad u_n \rightharpoonup  u^0\;\;\mbox{in}\;\;L^p(\Omega) \quad  \mbox{and}\quad u_n \to u^0\;\;\mbox{a.e. in}\;\;\Omega.
\end{equation}
Moreover, standard arguments ensure that the function $u_0\in X_0^s$ weakly solves problem $(P)$.\par
 Now, let $\psi_{n}^{1}$ be the function given by
$$
\psi_{n}^{1}(x) := \begin{cases}
(u_n - u^0)(x)&x\in \Omega\\
0&x\in \R^N \setminus \Omega .
\end{cases}
$$
By using (\ref{E22}) it follows that
$$
\psi_{n}^{1}\rightharpoonup 0\;\;\mbox{in $H^s(\R^N)$ and $L^p(\R^N)$}.
$$

\indent With the above notations we are able to prove the following facts:
\begin{step}\label{Eclaim3}
\begin{equation}\label{E25}
I_\infty(\psi_n^1) = I(\psi_{n}^{1}) + o_n(1) = I(u_n) - I(u^0) + o_n(1).
\end{equation}
\end{step}
%\textcolor{red}{Why $o_n(1)$ instead of $o(1)$?}
\begin{proof}
We notice that
\begin{equation}\label{E26}
\begin{aligned}
\|\psi_{n}^{1}\|_{s}^{2} %&= \int_{\R^N}\int_{\R^N} \frac{|\psi_{n}^{1}(x)-\psi_{n}^{1}(y)|^2}{|x-y|^{N+2s}}dy dx + \int_{\R^N}(\psi_{n}^{1}(x))^2dx\\
& = \left(\iint_{\mathcal{Q}} + \iint_{\mathcal{O}} \right) \frac{|\psi_{n}^{1}(x) - \psi_{n}^{1}(y)|^2}{|x-y|^{N+2s}}dy dx + \left(\int_{\Omega} + \int_{\R^N \setminus \Omega} \right)(\psi_{n}^{1}(x))^2dx\\
&=\|\psi_{n}^{1}\|^2 = \|u_n - u^0\|^2 = \|u_n\|^2 - 2\langle u_n,u^0\rangle + \|u^0\|^2\\
& = \|u_n\|^2 - \|u^0\|^2 + o_n(1)
\end{aligned}
\end{equation}
and
\begin{equation}\label{E27}
\|u_n\|_{L^p(\mathbb{R}^N)}^{p} = \|u^0\|_{L^p(\mathbb{R}^N)}^{p} + \|\psi_{n}^{1}\|_{L^p(\mathbb{R}^N)}^{p} + o_n(1).
\end{equation}
Relations (\ref{E26}) and (\ref{E27}) immediately yields \eqref{E25}.
\end{proof}

\begin{step}\label{Eclaim4}
$$
I_\infty'(\psi_{n}^{1}) = I'(\psi_{n}^{1}) + o_n(1) = I'(u_n) - I'(u^0) +o_n(1) = o_n(1).
$$
\end{step}
\begin{proof}
For each $v\in X_0^s$ with $\|v\| \leq 1$, one has
$$
I'(\psi_{n}^{1}) v = \iint_{\mathcal{Q}}\frac{[\psi_{n}^{1}(x) - \psi_{n}^{1}(y)][v(x)- v(y)]}{|x-y|^{N+2s}}dy dx + \int_{\Omega} \psi_{n}^{1}vdx - \int_{\Omega}|\psi_{n}^{1}|^{p-2}\psi_{n}^{1}vdx
$$
and
$$
\begin{aligned}
{I}'_{\infty}(\psi_{n}^{1})v &= \iint_{\R^{2N}} \frac{[\psi_{n}^{1}(x) - \psi_{n}^{1}(y)][v(x) - v(y)]}{|x-y|^{N+2s}}dy dx + \int_{\R^N} \psi_{n}^{1}vdx - \int_{\R^N} |\psi_{n}^{1}|^{p-2}\psi_{n}^{1}vdx\\
&=\iint_{\mathcal{Q}}\frac{[\psi_{n}^{1}(x) - \psi_{n}^{1}(y)][v(x) - v(y)]}{|x-y|^{N+2s}}dy dx + \int_{\Omega} \psi_{n}^{1}vdx - \int_{\Omega}|\psi_{n}^{1}|^{p-2}\psi_{n}^{1}vdx.
\end{aligned}
$$
Then
$$
|\langle I'(\psi_{n}^{1}) - I'_{\infty}(\psi_{n}^{1}), v \rangle| = \left| \int_{\mathbb{R}^N \setminus \Omega}\psi_{n}^{1}vdx - \int_{\mathbb{R}^N \setminus  \Omega}|\psi_{n}^{1}|^{p-2}\psi_{n}^{1} v\right|=o_n(1),
$$
for every $v\in X_0^s$ with $\|v\| \leq 1$. Consequently, it follows that
\begin{equation}\label{E41}
I'_{\infty}(\psi_{n}^{1}) = I'(\psi_{n}^{1}) + o_n(1).
\end{equation}
Now, we are going to show that
\begin{equation}\label{E42}
I'(\psi_{n}^{1}) = I'(u_n) - I'(u^0) + o_n(1) = o_n(1).
\end{equation}
Indeed, by definition of $\psi_{n}^{1}$, it is easy to see that
\begin{equation}\label{E43}
\begin{aligned}
&\langle I'(\psi_{n}^{1}) - I'(u_n) + I'(u^0), v\rangle\\
 & = \iint_{\mathcal{Q}} \frac{[\psi_{n}^{1}(x) - \psi_{n}^{1}(y)][v(x) -v(y)]}{|x-y|^{N+2s}}dy dx + \int_{\Omega} \psi_{n}^{1}vdx - \int_{\Omega}|\psi_{n}^{1}|^{p-2}\psi_{n}^{1}vdx\\
&-\iint_{\mathcal{Q}}\frac{[u_n(x) - u_n(y)][v(x) - v(y)]}{|x-y|^{N+2s}}dy dx - \int_{\Omega} u_nvdx + \int_{\Omega}|u_n|^{p-2}u_nvdx\\
&+\iint_{\mathcal{Q}}\frac{[u^0(x) - u^0(y)][v(x)-v(y)]}{|x-y|^{N+2s}}dy dx + \int_{\Omega} u^0vdx - \int_{\Omega}|u^0|^{p-2}u^0vdx\\
&=\int_{\Omega}( |u_n|^{p-2}u_n -|u^0|^{p-2}u^0 - |\psi_{n}^{1}|^{p-2}\psi_{n}^{1}) v \,dx.
\end{aligned}
\end{equation}
On the other hand, bearing in mind that
\begin{equation}\label{E44}
\left( \int_{\Omega}\left||u_n|^{p-2}u_n - |u^0|^{p-2}u^0 - |\psi_{n}^{1}|^{p-2}\psi_{n}^{1}\right|^{\frac{p}{p-1}}dx \right)^{\frac{p-1}{p}} = o_n(1),
\end{equation}
relation \eqref{E43} directly yields \eqref{E42}.
\end{proof}

\indent
If $\psi_{n}^{1} \to 0$ in $X_0^s$ the statements of the main result are verified.
Thus, we can suppose that
\begin{equation}\label{E46}
\psi_{n}^{1} \not \to 0\;\;\mbox{in}\;\;X_0^s.
\end{equation}
By using the fact that
$$
I_\infty(\psi_{n}^{1}) = \frac{1}{2}\|\psi_{n}^{1}\|_s^2 - \frac{1}{p}\int_{\R^N}|\psi_{n}^{1}|^pdx,
$$
and $I'_{\infty}(\psi_{n}^{1}) = o_n(1)$, we have
\begin{equation}\label{E45}
I'_{\infty}(\psi_{n}^{1})\psi_{n}^{1} = \|\psi_{n}^{1}\|_{s}^{2} - \int_{\R^N} |\psi_{n}^{1}|^pdx = o_n(1).
\end{equation}
Then
$$
I_{\infty}(\psi_{n}^{1}) = \frac{1}{2}\|\psi_{n}^{1}\|_{s}^{2} - \frac{1}{p}\|\psi_{n}^{1}\|_{s}^{2} +o_n(1) = \left(\frac{1}{2} - \frac{1}{p}\right)\|\psi_{n}^{1}\|_{s}^{2} + o_n(1).
$$
By (\ref{E46}), there is $\alpha >0$ such that
\begin{equation}\label{E47}
I_{\infty}(\psi_{n}^{1})\geq \alpha >0.
\end{equation}
Now, let us decompose $\R^N$ into $N$-dimensional unit hypercubes $Q_i$ whose vertices have integer coordinates and put
\begin{equation}\label{E48}
d_n = \max_{i}\|\psi_{n}^{1}\|_{L^p(Q_i)}.
\end{equation}
Arguing as in \cite{VBGC}, there is $\gamma >0$ such that
\begin{equation}\label{E49}
d_n \geq \gamma >0.
\end{equation}
Denote by $\{y_n^1\}$ the center of a hypercube $Q_i$ in which $\|\psi_{n}^{1}\|_{L^p(Q_i)} = d_n$ and let us prove that this sequence is unbounded in $\R^N$. Arguing by contradiction, suppose that the sequence $\{y_n^1\}$ is bounded in $\R^N$. Then, there is $R>0$ such that
\begin{equation}\label{E50}
\int_{B(0,R)}|\psi_{n}^{1}|^pdx \geq \int_{Q_i(y_n^1)}|\psi_{n}^{1}|^pdx = d_{n}^{p} > \gamma^p >0.
\end{equation}
On the other hand, since $\psi_{n}^{1} \rightharpoonup  0$ in $H^s(\R^N)$, Lemma \ref{sobolev} gives %we get that $\psi_{n}^{1}  \to 0$ in $L_{loc}^{p}(\R^N)$, namely
$$
\int_{B(0,R)}|\psi_{n}^{1} |^pdx \to 0,\;\;\mbox{as}\;\;n\to +\infty,
$$
against (\ref{E50}). Therefore, the sequence $\{y_n^1\}$ is an unbounded. Since
$$
\|\psi_{n}^{1} (\cdot + y_{n}^{1})\|_{s} = \|\psi_{n}^{1} \|_s \quad \forall n \in \mathbb{N},
$$
we deduce that $\{\psi_{n}^{1} (\cdot + y_n^1)\}$ is a bounded sequence in $H^s(\R^N)$. Then, there is $u^1\in H^s(\R^N)$ such that
$$
\psi_{n}^{1} (\cdot + y_n^1) \rightharpoonup  u^1\;\;\mbox{in}\;\;H^s(\R^N)\\
$$
and
$$
\psi_{n}^{1} (\cdot + y_n^1) \to u^1\;\;\mbox{in}\;\;L_{loc}^{p}(\R^N).
$$
\begin{step}\label{Eclaim5}
$u^1$ is a nontrivial weak solution of $(P_\infty)$.
\end{step}
\begin{proof}
First of all, by (\ref{E50}), we derive that $u^1\neq 0$, and by a straightforward computation
\begin{equation*}\label{E52}
I'_{\infty}(\psi_{n}^{1}(\cdot + y_n^1))v = o_n(1),\;\;\forall v\in C_{0}^{\infty}(\R^N).
\end{equation*}
Then, taking the limit of $n \to +\infty$, we find
\begin{equation*}\label{E53}
\int_{\R^N}\int_{\R^N}\frac{[u^1(x) -u^1(y)][v(x) - v(y)]}{|x-y|^{N+2s}}dydx +\int_{\R^N} u^1vdx = \int_{\R^N} |u^1|^{p-2}u^1vdx,\;\;\forall v\in C_{0}^{\infty}(\Omega).
\end{equation*}
Now, the density of  $C_{0}^{\infty}(\R^N)$ in $H^s(\R^N)$ gives
$$
\int_{\R^N}\int_{\R^N}\frac{[u^1(x) -u^1(y)][w(x) - w(y)]}{|x-y|^{N+2s}}dydx +\int_{\R^N} u^1wdx = \int_{\R^N} |u^1|^{p-2}u^1wdx,\;\;\forall w\in H^s(\R^N),
$$
from where it follows that $u^1$ is a nontrivial solution of ($P_\infty$).
\end{proof}
We can repeat this process to obtain the sequence
$$
\psi_{n}^{j}(x) = \psi_{n}^{j-1}(x + y_n^{j-1})-u^{j-1}(x),\;\;j\geq 2,
$$
where
$$
|y_n^j|\to +\infty,\;\;\mbox{as}\;\;n\to +\infty
$$
and
\begin{equation}\label{E54}
\psi_{n}^{j-1}(x + y_n^{j-1}) \rightharpoonup u^{j-1}\;\;\mbox{in}\;\;H^s(\R^N),
\end{equation}
where each $u^j$ is a nontrivial solution of ($P_\infty$). Now, by induction, we have the following equalities
\begin{equation}\label{E55}
\begin{aligned}
\|\psi_{n}^{j}\|_{s}^{2} &= \|\psi_{n}^{j-1}\|_{s}^{2} - \|u^{j-1}\|_{s}^{2} + o_n(1) = \|u_n\|^2 - \|u^0\|^2 - \sum_{i=1}^{j-1}\|u^i\|_{s}^{2} + o_n(1)
\end{aligned}
\end{equation}
and
\begin{equation}\label{E56}
\begin{aligned}
I_{\infty}(\psi_{n}^{j}) &= I_{\infty}(\psi_{n}^{j-1}) - I_{\infty}(u^{j-1}) + o_n(1)=I(u_n) - I(u^0) - \sum_{i=1}^{j-1}\tilde{I}(u^i) + o_n(1).
\end{aligned}
\end{equation}

Since $u^j$ is a nontrivial solution of ($P_\infty$), it follows that
\begin{equation}\label{E58}
\|u^j\|_{s}^{2} \geq M_{\infty}^{\frac{p}{p-2}},
\end{equation}
for every $1\leq j\leq k$.
Now, arguing as in \cite{VBGC}, we deduce that above argument will stop after a finite number of steps.

\end{proof}

\begin{corollary}\label{Ecor01}
Let $\{u_n\}$ be as in Lemma \ref{Elm01} and let
\begin{equation}\label{E59}
c< \left(\frac{1}{2}-\frac{1}{p} \right)M_\infty^{\frac{p}{p-2}}.
\end{equation}
Then $\{u_n\}$ admits a strongly convergent subsequence. Hence,  the functional $I$ verifies the $(PS)_c$ condition, for every $$c \in \left(0,\left(\frac{1}{2}-\frac{1}{p} \right)M_\infty^{\frac{p}{p-2}} \right).$$
\end{corollary}
\begin{proof}
By our hypotheses, one has
$$
I(u_n)\to c\;\;\mbox{and}\;\;I'(u_n)\to 0\;\;\mbox{as}\;\;n\to +\infty,
$$
where $c$ satisfies (\ref{E59}). Without loss of generality, we can suppose that $\{u_n\}$ is bounded in $X_0^s$. Then, up to some subsequence,  there exists $u^0\in X_0^s$ such that
$$
u_n \rightharpoonup u^0\;\;\mbox{in}\;\;X_0^s.
$$
If $u \not \to u^0$ in $X_0^s$, by Lemma \ref{Elm01} we must have $k \geq 1$. Hence,
$$
I(u_n)\to c \geq \left( \frac{1}{2}-\frac{1}{p}\right)M_{\infty}^{\frac{p}{p-2}},
$$
which is a contradiction with (\ref{E59}). Therefore
$$
u_n \rightharpoonup u^0\;\;\mbox{and}\;\;\|u_n\|^2 \to \|u^0\|^2,
$$
which implies that $u_n \to u^0$ in $X_0^s$.
\end{proof}

\begin{corollary}\label{Ecor02}
Let $\mathcal{P}$ the set of nonnegative functions in $X_0^s$. Assume that there is $\{u_n\} \subset \mathcal{P}$ that satisfies the assumptions of Lemma \ref{Elm01}. If
\begin{equation}\label{E62}
\left( \frac{1}{2} - \frac{1}{p}\right)M^{\frac{p}{p-2}} < c < 2\left( \frac{1}{2} - \frac{1}{p}\right) M^{\frac{p}{p-2}},
\end{equation}
then $\{u_n\}$ has a strongly convergent subsequence. Hence, the energy functional $I$ satisfies the $(PS)_c$ condition, for every $$c \in \left( \left( \frac{1}{2} - \frac{1}{p}\right)M^{\frac{p}{p-2}}, 2\left( \frac{1}{2} - \frac{1}{p}\right) M^{\frac{p}{p-2}} \right).$$
\end{corollary}

\begin{proof}
As in Corollary \ref{Ecor01}, there is $u^0\in X_0^s$ such that $u_n \rightharpoonup u^0$ in $X_0^s$. Assume that $u_n \not \to u^0$ in $X_0^s$, then we must have $k\geq 1$ in Lemma \ref{Elm01}.  As $M_\infty = M$, for $k\geq 2$ we get
$$
I(u_n)\to c \geq \left( \frac{1}{2}-\frac{1}{p}\right)[M_{\infty}^{\frac{p}{p-2}} + M^{\frac{p}{p-2}}]=2\left( \frac{1}{2} - \frac{1}{p}\right) M^{\frac{p}{p-2}},
$$
which is impossible, consequently $k$ cannot be greater than $1$.  If $u^0 = 0$, we have that
$u^1$ is a positive ground state solution of $(P_\infty)$, which is unique and satisfies
$$
I_{\infty}(u^1) = \left(\frac{1}{2}-\frac{1}{p}\right)M_{\infty}^{\frac{p}{p-2}}.
$$
Then
\begin{equation}\label{E63}
I(u_n) = I(u^0) + I_{\infty}(u^1) + o_n(1) = \left(\frac{1}{2}-\frac{1}{p} \right)M_{\infty}^{\frac{p}{p-2}} + o_n(1)=\left(\frac{1}{2}-\frac{1}{p} \right)M^{\frac{p}{p-2}} + o_n(1),
\end{equation}
contrary to (\ref{E62}). Thereby $u^0\neq 0$, and we must have
$$
I(u^0) \geq \left(\frac{1}{2}-\frac{1}{p} \right)M^{\frac{p}{p-2}} \quad \mbox{and} \quad I_\infty(u^1) \geq \left(\frac{1}{2}-\frac{1}{p}\right)M_{\infty}^{\frac{p}{p-2}},
$$
leading to
$$
I(u_n) \geq \left(\frac{1}{2}-\frac{1}{p} \right)[M_{\infty}^{\frac{p}{p-2}} + M^{\frac{p}{p-2}}] + o_n(1)=2\left( \frac{1}{2} - \frac{1}{p}\right) M^{\frac{p}{p-2}}+ o_n(1),
$$
which contradicts (\ref{E62}). From this, we cannot have $k = 1$, and so,
$$
\|u_n\|^{2} \to \|u^0\|^{2}\;\;\mbox{as}\;\;n\to +\infty,
$$
that is, $u_n \to u^0$ in $X_0^s$ and $u^0\neq 0$.
\end{proof}

In the sequel, let us consider the set
$$
\mathcal{V} := \left\{u\in X_0^s:\;\;\int_{\Omega}|u|^pdx = 1 \right\}
$$
and the functional $J:X_0^s \to \R$ defined by
\begin{equation}\label{fun01}
J(u) := \|u\|^2.
\end{equation}
Moreover, we consider the norm
$$
\|J'(u)\|_{*} = \sup_{w\in T_u\mathcal{V}, \|w\|\leq 1} |\langle J'(u), w\rangle|
$$
where
$$
T_u\mathcal{V} := \{v\in X_0^s:\;\;G'(u)v = 0\}
$$
and $G: X_0^s\to \R$ is the functional given by
\begin{equation}\label{fun02}
G(u) := \int_{\Omega}|u|^pdx.
\end{equation}
\begin{corollary}\label{Ecor03}
$J$ satisfies the Palais-Smale condition in
$$
\mathcal{Z} := \mathcal{P}\cap \mathcal{V} \cap \{u\in X_0:\;\;M < J(u) < 2^{\frac{p-2}{p}}M \}.
$$
\end{corollary}
\begin{proof}
Let $\{u_n\} \subset \mathcal{Z}$  be a sequence  satisfying
\begin{equation}\label{E64}
J(u_n) \to c \in \left(M, 2^{\frac{p-2}{p}}M \right)\;\;\mbox{and}\;\; \|J'(u_n)\|_{*} \to 0,\;\;\mbox{as}\;\;n\to +\infty.
\end{equation}
Setting $v_n = c^{\frac{1}{p-2}}u_n$ and
$$
d = \left( \frac{1}{2}-\frac{1}{p}\right)c^{\frac{p}{p-2}} \in  \left( \left( \frac{1}{2}-\frac{1}{p}\right) M^{\frac{p}{p-2}}, 2\left( \frac{1}{2}-\frac{1}{p}\right)M^{\frac{p}{p-2}} \right),
$$
we derive that
$$
\begin{aligned}
I(v_n) & = \frac{1}{2}\left(\iint_{\mathcal{Q}}\frac{|v_n(x) - v_n(y)|^2}{|x-y|^{N+2s}}dy dx + \int_{\Omega} v_n^2(x)dx \right) - \frac{1}{p}\int_{\Omega}|v_n|^pdx\\
& = \frac{c^{\frac{2}{p-2}}}{2}\|u_n\|^2 - \frac{c^{\frac{p}{p-2}}}{2} \to d.
\end{aligned}
$$
Now we claim that $I'(v_n) \to 0$ as $n\to +\infty$. Indeed, by Proposition 5.12 in \cite{MW},
$$
\|J'(u_n)\|_{*} = \min_{\lambda \in \R}\|J'(u_n) - \lambda G'(u_n)\| = \|J'(u_n) - \lambda_nG'(u_n)\|.
$$
Hence, by (\ref{E64}),
\begin{equation}\label{E65}
J'(u_n) - \lambda_nG'(u_n) = o_n(1)\;\;\mbox{in}\;\;(X_0^s)^*.
\end{equation}
Therefore, by a straightforward  computation,
$$
I'(v_n) = o_n(1)\;\;\mbox{in}\;\;(X_0)^*.
$$
Using the above limit, we can apply Corollary \ref{Ecor02} to deduce that $\{v_n\}$ has a subsequence convergent, which implies that $\{u_n\}$ has a subsequence convergent and so $J$ satisfies the (PS) condition in $\mathcal{Z}$.
\end{proof}

\section{Proof of some estimates}

We start this section by introducing  the following operator
$$
\begin{aligned}
\phi_\rho : \R^N &\to H^s(\R^N)\\
y&\mapsto\phi_\rho (y) = \frac{v_{y}^{\rho}(x)}{\|v_{y}^{\rho}\|_{L^p(\R^N)}} ,
\end{aligned}
$$
where
$$
v_{y}^{\rho}(x) := \zeta(x)\varphi(x-y) = \xi\left(\frac{|x|}{\rho}\right)\varphi (x-y),
$$
and $\zeta, \xi, \varphi$ given as in the proof of Theorem \ref{Etm01}. A direct computation ensures that the functions $\phi_\rho, v_{y}^{\rho}$ belong to $X_0^s$ and $L^p(\Omega)$.

\begin{lemma}\label{Tlm01}
The following relations hold:
\begin{itemize}
\item[$(i)$] For every $y\in\R^N$
$$\phi_\rho(y) \to \varphi (\cdot - y)$$
in $H^s(\R^N)$, as $\rho \to 0$$;$
\item[$(ii)$] $\|\phi_\rho\|_{s}^{2} \to M$ for each $\rho$, as $|y|\to +\infty$.
\end{itemize}
\end{lemma}

\begin{proof}
Part {$(i)$} - Taking into account that $\xi({|x|}/{\rho}) = 1$ for every $|x|\geq 2\rho$, since $\xi$ is bounded and $\varphi$ is radially symmetric and non-increasing, we have that
$$
\begin{aligned}
\|v_{y}^{\rho} - \varphi(\cdot - y)\|_{L^p(\R^N)}^{p} &= \int_{\R^N}\left| \xi\left( \frac{|x|}{\rho}\right)\varphi(x-y) - \varphi (x-y) \right|^pdx\\
&=\int_{B(0, 2\rho)}\left| \left( \xi\left( \frac{|x|}{\rho}\right) - 1 \right)\varphi (x-y) \right|^pdx\\
&\leq K_1\int_{B(0,2\rho)}|\varphi (x-y)|^pdx\leq K_1\varphi^p (0)|B(0,2\rho)|,
\end{aligned}
$$
for some $K_1>0$.\par
\noindent Therefore, taking the limit of $\rho \to 0$, for every $y\in\R^N$, we get
$$
\|v_{y}^{\rho} - \varphi(\cdot - y)\|_{L^p(\R^N)} \to 0,
$$
and so,
\begin{equation}\label{T01-0}
\|v_{y}^{\rho}\|_{L^p(\R^N)} \to \|\varphi(\cdot - y)\|_{L^p(\R^N)} = \|\varphi\|_{L^p(\R^N)} =1,
\end{equation}
for every $y\in\R^N$.\par
\noindent On the other hand,
$$
\begin{aligned}
\|v_{y}^{\rho} - \varphi(\cdot-y)\|_{s}^{2} &= \int_{\R^N}\int_{\R^N}\frac{|(v_{y}^{\rho} - \varphi(\cdot-y))(x) - (v_{y}^{\rho} - \varphi(\cdot-y))(z)|^2}{|x-z|^{N+2s}}dz dx \\
&+ \int_{\R^N} |v_{y}^{\rho}(x) - \varphi(x-y)|^2dx.
\end{aligned}
$$
Of course, we also have
\begin{equation}\label{T01}
\int_{\R^N} \left| \xi\left(\frac{|x|}{\rho} \right)\varphi (x-y) - \varphi (x-y)\right|^2dx \leq  K_1^2\varphi^2(0)|B(0, 2\rho)| = K_2\rho^N.
\end{equation}
Setting
$$
{I_1} := \iint_{\R^{2N}}\frac{\left|\left(\xi(\frac{|x|}{\rho})-\xi(\frac{|z|}{\rho})\right)\right|^2|\varphi(x-y) |^2}{|x-z|^{N+2s}}dz dx
$$
and
$$
{I_2}:= \iint_{\R^{2N}}\frac{\left|\xi(\frac{|z|}{\rho}) - 1 \right|^2|\varphi (x-y) - \varphi(z-y))|^2}{|x-z|^{N+2s}}dz dx,
$$
the following inequality holds
$$
\begin{aligned}
& \int_{\R^N}\int_{\R^N}\frac{|(v_{y}^{\rho} - \varphi(\cdot-y))(x) - (v_{y}^{\rho} - \varphi(\cdot-y))(z)|^2}{|x-z|^{N+2s}}dz dx
%&=  \int_{\R^N}\int_{\R^N}\frac{|\xi(\frac{|x|}{\rho})\varphi(x-y) - \varphi(x-y) - (\xi(\frac{|z|}{\rho})\varphi (z-y) - \varphi(z-y))|^2}{|x-z|^{N+2s}}dz dx\\
%&=\int_{\R^N}\int_{\R^N}\frac{|\xi(\frac{|x|}{\rho})\varphi(x-y) - \xi(\frac{|z|}{\rho})\varphi (z-y) - (\varphi (x-y) - \varphi(z-y))|^2}{|x-z|^{N+2s}}dz dx\\
%&=\int_{\R^N}\int_{\R^N}\frac{|\left(\xi(\frac{|x|}{\rho})-\xi(\frac{|z|}{\rho})\right)\varphi(x-y) + \left(\xi(\frac{|z|}{\rho}) - 1 \right)(\varphi (x-y) - \varphi(z-y))|^2}{|x-z|^{N+2s}}dz dx\\
%&\leq 2\iint_{\R^{2N}}\frac{\left|\left(\xi(\frac{|x|}{\rho})-\xi(\frac{|z|}{\rho})\right)\right|^2|\varphi(x-y) |^2}{|x-z|^{N+2s}}dz dx + 2\iint_{\R^{2N}}\frac{\left|\xi(\frac{|z|}{\rho}) - 1 \right|^2|\varphi (x-y) - \varphi(z-y))|^2}{|x-z|^{N+2s}}dz dx
 \leq {I_1} + {I_2}.
 \end{aligned}
$$
\noindent After a change of variables
$$
{I_2}=\iint_{\R^{2N}}\frac{\left|\xi(\frac{|z+y|}{\rho}) - 1 \right|^2|\varphi (x) - \varphi(z))|^2}{|x-z|^{N+2s}}dz dx.
$$
Moreover, by definition of $\xi$ we also have
$$
\left|\xi\left(\frac{|z+y|}{\rho}\right) - 1 \right|^2\frac{|\varphi (x)-\varphi (z)|^2}{|x-z|^{N+2s}} \leq 4 \frac{|\varphi (x) - \varphi (z)|^2}{|x-z|^{N+2s}} \in L^1(\R^N \times \R^N)
$$
and
$$
\left|\xi\left(\frac{|z+y|}{\rho}\right) - 1 \right|^2\frac{|\varphi (x)-\varphi (z)|^2}{|x-y|^{N+2s}} \to 0 \;\;\mbox{a.e. in}\;\;\R^N \times \R^N
$$
as $\rho \to 0$.\par
\noindent Hence, the Lebesgue's theorem ensures that
\begin{equation}\label{T02-0}
{I_2} =  \iint_{\R^{2N}}\frac{\left|\xi(\frac{|z+y|}{\rho}) - 1 \right|^2|\varphi (x) - \varphi(z))|^2}{|x-z|^{N+2s}}dz dx \to 0\;\;\mbox{as}\;\;\rho \to 0,
\end{equation}
for every $y\in \R^N$.\par
\noindent Now, following \cite{MXBZXZ} we show that, for every $y\in \R^N$, one has
\begin{equation}\label{T02}
{I_1} = \iint_{\R^{2N}}\frac{\left|\left(\xi(\frac{|x|}{\rho})-\xi(\frac{|z|}{\rho})\right)\right|^2|\varphi(x-y) |^2}{|x-z|^{N+2s}}dz dx \to 0 \;\;\mbox{as}\;\;\rho \to 0.
\end{equation}
Indeed, after a change of variables, it follows that
\begin{equation}\label{T03}
{I_1} = \int_{\R^N}\int_{\R^N} |\varphi (x)|^2\frac{\left|\xi (\frac{x+y}{\rho}) - \xi(\frac{|z+y|}{\rho}) \right|^2}{|x-z|^{N+2s}}dz dx.
\end{equation}
Now, we decompose $\mathbb{R}^N \times \mathbb{R}^N$ as follows
$$
\R^N \times \R^N := \Omega_1 \cup \Omega_2 \cup \Omega_3,
$$
where
$$
\begin{aligned}
&\Omega_1 := \R^N \setminus B(-y,2\rho) \times \R^N \setminus B(-y, 2\rho)\\
&\Omega_2 :=B(-y, 2\rho)\times \R^N\\
&\Omega_3 := \R^N \setminus B(-y,2\rho) \times B(-y,2\rho).
\end{aligned}
$$
Then
\begin{equation}\label{T04}
{I_1} = \left(\iint_{\Omega_1} + \iint_{\Omega_2} + \iint_{\Omega_3} \right) |\varphi (x)|^2\frac{\left| \xi \left(\frac{|x+y|}{\rho} \right)- \xi \left(\frac{|z+y|}{\rho}\right)\right|^{2}}{|x-z|^{N+2s}}dz dx.
\end{equation}

\noindent
{\bf Case 1:} Let $(x,z)\in \Omega_1.$ Since
$$
\xi \left(\frac{|x+y|}{\rho}\right) = \xi \left(\frac{|z+y|}{\rho}\right) = 0,\;\;\forall x, z\in \R^N \setminus B(-y,2\rho),
$$
one has
\begin{equation}\label{T05}
\iint_{\Omega_1} |\varphi(x)|^2\frac{\left|\xi \left(\frac{|x+y|}{\rho}\right) - \xi \left(\frac{|z+y|}{\rho}\right) \right|^2}{|x-z|^{N+2s}}dz dx = 0.
\end{equation}

\noindent
{\bf Case 2:} Let $(x,z)\in \Omega_2$. We notice that
%$$
%|z+y| \leq |z-x + x+y| \leq |z-x| + |x+y|\leq 3\rho
%$$
%Moreover we note that
\begin{equation}\label{T06}
\begin{aligned}
\iint_{\Omega_2} |\varphi (x)|^2\frac{\left|\xi \left(\frac{|x+y|}{\rho}\right) - \xi \left(\frac{|z+y|}{\rho}\right) \right|^2}{|x-z|^{N+2s}} dz dx &
\end{aligned}
\end{equation}
$$
\qquad\qquad=\int_{B(-y,2\rho)}\int_{\overline{B}(x,\rho)} |\varphi (x)|^2\frac{\left|\xi \left(\frac{|x+y|}{\rho}\right) - \xi \left(\frac{|z+y|}{\rho}\right) \right|^2}{|x-z|^{N+2s}} dz dx
$$
$$
\qquad\qquad\qquad+ \int_{B(-y,2\rho)}\int_{\R^N \setminus \overline{B}(x,\rho)}|\varphi (x)|^2\frac{\left|\xi \left(\frac{|x+y|}{\rho}\right) - \xi \left(\frac{|z+y|}{\rho}\right) \right|^2}{|x-z|^{N+2s}} dz dx.
$$
Now, a simple computation ensures that
\begin{equation}\label{T07}
\int_{B(-y,2\rho)}\int_{\overline{B}(x,\rho)}|\varphi (x)|^2\frac{\left|\xi \left(\frac{|x+y|}{\rho}\right) - \xi \left(\frac{|z+y|}{\rho}\right) \right|^2}{|x-z|^{N+2s}}dz dx
\end{equation}
$$
\qquad\qquad\qquad\quad\quad\leq \frac{K}{2-2s}\frac{|S^{N-1}|}{\rho^{2s}}\int_{B(-y,2\rho)}|\varphi (x)|^2dx,
$$
for some $K>0$.\par
\noindent On the other hand, we also have
\begin{equation}\label{T08}
\int_{B(-y,2\rho)}\int_{\R^N \setminus \overline{B}(x,\rho)}|\varphi (x)|^2\frac{\left|\xi \left(\frac{|x+y|}{\rho}\right) - \xi \left(\frac{|z+y|}{\rho}\right) \right|^2}{|x-z|^{N+2s}} dz dx
\end{equation}
$$
\qquad\qquad\qquad\quad\quad\leq \frac{2}{s}\frac{|S^{N-1}|}{\rho^{2s}}\int_{B(-y,2\rho)}|\varphi (x)|^2dx.
$$
Thereby, by (\ref{T07}) and (\ref{T08}), it follows that
\begin{equation}\label{T09}
\int_{B(-y,2\rho)}\int_{\R^N} |\varphi (x)|^2\frac{\left|\xi \left(\frac{|x+y|}{\rho}\right) - \xi \left(\frac{|z+y|}{\rho}\right) \right|^2}{|x-z|^{N+2s}} dz dx \leq \frac{C_1}{\rho^{2s}}\int_{B(-y,2\rho)}|\varphi (x)|^2dx,
\end{equation}
for some $C_1>0$.\par

\noindent
{\bf Case 3:} Let $(x,z) \in \Omega_3$.
Let us denote
$$
\mathcal{B}(x,\rho) := \{z\in B(-y,2\rho):\;|x-z|\leq \rho\}\;\;
$$
and
$$
\mathcal{B}^c(x,\rho) := \{z\in B(-y,2\rho):\;|x-z|> \rho\}.
$$
We have
$$
\begin{aligned}
\int \int_{\Omega_3} |\varphi (x)|^2& \frac{\left|\xi \left(\frac{|x+y|}{\rho}\right) - \xi \left(\frac{|z+y|}{\rho}\right) \right|^2}{|x-z|^{N+2s}} dz dx\\
& = \int_{\R^N \setminus B(-y,2\rho)}\int_{\mathcal{B}(x,\rho)}|\varphi (x)|^2\frac{\left|\xi \left(\frac{|x+y|}{\rho}\right) - \xi \left(\frac{|z+y|}{\rho}\right) \right|^2}{|x-z|^{N+2s}} dz dx\\
& + \int_{\R^N \setminus B(-y,2\rho)} \int_{\mathcal{B}^c(x,\rho)}|\varphi (x)|^2\frac{\left|\xi \left(\frac{|x+y|}{\rho}\right) - \xi \left(\frac{|z+y|}{\rho}\right) \right|^2}{|x-z|^{N+2s}} dz dx.
\end{aligned}
$$
Arguing as above, it is easy to see that
\begin{equation}\label{T10}
\int_{\R^N \setminus B(-y,2\rho)}\int_{\mathcal{B}(x,\rho)}|\varphi (x)|^2 \frac{\left|\xi \left(\frac{|x+y|}{\rho}\right) - \xi \left(\frac{|z+y|}{\rho}\right) \right|^2}{|x-z|^{N+2s}} dz dx
\end{equation}
$$
\qquad\qquad\qquad\quad\quad\leq \frac{K|S^{N-1}|}{\rho^{2s}}\int_{B(-y,3\rho)} |\varphi (x)|^2dx,
$$
for some $K>0$.\par
\noindent Now, let us consider the integral
 $$
 \int_{\R^N \setminus B(-y,2\rho)} \int_{\mathcal{B}^c(x,\rho)}|\varphi (x)|^2\frac{\left|\xi \left(\frac{|x+y|}{\rho}\right) - \xi \left(\frac{|z+y|}{\rho}\right) \right|^2}{|x-z|^{N+2s}} dz dx.
 $$
 If $(x,z)\in \R^N \setminus B(-y,2\rho)\times B(-y,2\rho)$, $|x-z|> \rho$  and $k>4$, we have
$$
\R^N \setminus B(-y,2\rho)\times B(-y,2\rho) \subset B(-y, k\rho)\times B(-y,2\rho) \cup [\R^N \setminus B(-y, k\rho) \times B(-y2\rho)].
$$
Therefore,
$$
\begin{aligned}
\int_{\R^N \setminus B(-y,2\rho)} \int_{\mathcal{B}^c(x,\rho)}|\varphi (x)|^2&\frac{\left|\xi \left(\frac{|x+y|}{\rho}\right) - \xi \left(\frac{|z+y|}{\rho}\right) \right|^2}{|x-z|^{N+2s}} dz dx\\
&\leq \int_{B(-y,k\rho)}\int_{\mathcal{B}^c(x,\rho)}|\varphi (x)|^2\frac{\left|\xi \left(\frac{|x+y|}{\rho}\right) - \xi \left(\frac{|z+y|}{\rho}\right) \right|^2}{|x-z|^{N+2s}} dz dx\\
&+ \int_{\R^N \setminus B(-y, k\epsilon)}\int_{\mathcal{B}^c(x,\rho)}|\varphi (x)|^2\frac{\left|\xi \left(\frac{|x+y|}{\rho}\right) - \xi \left(\frac{|z+y|}{\rho}\right) \right|^2}{|x-z|^{N+2s}} dz dx.
\end{aligned}
$$
Moreover, a direct computation gives
\begin{equation}\label{T11}
\int_{B(-y,k\rho)}\int_{\mathcal{B}^c(x,\rho)}|\varphi (x)|^2\frac{\left|\xi \left(\frac{|x+y|}{\rho}\right) - \xi \left(\frac{|z+y|}{\rho}\right) \right|^2}{|x-z|^{N+2s}} dz dx
\end{equation}
$$
\qquad\qquad\qquad\quad\quad\leq \frac{2}{s}\frac{|S^{N-1}|}{\rho^{2s}}\int_{B(-y,k\rho)}|\varphi (x)|^2dx.
$$
Now, if $(x,z)\in \R^N \setminus B(-y,k\epsilon)\times B(-y,2\rho)$, we must have
$$
|x-z| \geq |x+y| - |z+y| = \frac{|x+y|}{2} + \frac{|x+y|}{2} - 2\rho\geq \frac{|x+y|}{2} + \frac{k\rho}{2} - 2\rho \geq \frac{|x+y|}{2}.
$$
Hence, the H\"older inequality yields
\begin{equation}\label{T12}
\begin{aligned}
\int_{\R^N \setminus B(-y, k\epsilon)}&\int_{\mathcal{B}^c(x,\rho)}|\varphi (x)|^2\frac{\left|\xi \left(\frac{|x+y|}{\rho}\right) - \xi \left(\frac{|z+y|}{\rho}\right) \right|^2}{|x-z|^{N+2s}} dz dx\\
%&\leq 4 \int_{\R^N \setminus B(-y, k\epsilon)}\int_{\mathcal{B}^c(x,\rho)} \frac{|\varphi (x)|^2}{|x-z|^{N+2s}}dz dx\\
%&\leq 2^{2+N+2s}\int_{\R^N \setminus B(-y, k\epsilon)}|\varphi (x)|^2\int_{B(-y,2\rho)} \frac{dz}{|x+y|^{N+2s}}dx\\
& = 2^{2+N+2s}|S^{N-1}|2^N\rho^N \int_{\R^N\setminus B(-y,k\rho)} \frac{|\varphi (x)|^2}{|x+y|^{N+2s}}dx\\
&\leq 2^{2+2N+2s}|S^{N-1}|\rho^N \left( \int_{\R^N \setminus B(-y,k\rho)}|\varphi (x)|^{2_{s}^{*}}dx \right)^{\frac{2}{2_{s}^{*}}}\left( \int_{\R^N \setminus B(-y,k\rho)} \frac{dx}{|x+y|^{\frac{N^2}{2s} + N}}\right)^{\frac{2s}{N}}\\
& = 2^{2+2N+2s}|S^{N-1}|\left( \frac{2s}{N^2}|S^{N-1}|\right)^{\frac{2s}{N}} \frac{1}{k^N}\left(\int_{\R^N\setminus B(-y, k\rho)}|\varphi (x)|^{2_{s}^{*}}dx \right)^{\frac{2}{2_{s}^{*}}}.
\end{aligned}
\end{equation}
Therefore, by (\ref{T11}) and (\ref{T12}), it follows that
\begin{equation}\label{T13}
\begin{aligned}
\int_{\R^N \setminus B(-y,2\rho)} &\int_{\mathcal{B}^c(x,\rho)}|\varphi (x)|^2\frac{\left|\xi \left(\frac{|x+y|}{\rho}\right) - \xi \left(\frac{|z+y|}{\rho}\right) \right|^2}{|x-z|^{N+2s}} dz dx\\
& \leq \frac{C_2}{\rho^{2s}}\int_{B(-y,3\rho)}|\varphi (x)|^2dx + \frac{C_3}{\rho^{2s}} \int_{B(-y,k\rho)}|\varphi (x)|^2dx + \frac{C_4}{k^N}\left( \int_{\R^N \setminus B(-y,k\rho)} |\varphi (x)|^{2_{s}^{*}}\right)^{\frac{2}{2_{s}^{*}}},
\end{aligned}
\end{equation}
for some positive constants $C_i$, $i=2,3,4$.\par
\noindent Finally, by (\ref{T05}), (\ref{T09}) and (\ref{T13}), one has
$$
\begin{aligned}
{I_1} %&= \int_{\R^N}\int_{\R^N} |\varphi (x)|^2\frac{\left|\xi (\frac{x+y}{\rho}) - \xi(\frac{|z+y|}{\rho}) \right|^2}{|x-z|^{N+2s}}dz dx\\
&\leq \frac{C_1}{\rho^{2s}}\int_{B(-y,2\rho)}|\varphi (x)|^2dx + \frac{C_2}{\rho^{2s}}\int_{B(-y,3\rho)}|\varphi (x)|^2dx + \frac{C_3}{\rho^{2s}}\int_{B(-y, k\rho)}|\varphi (x)|^2dx\\
&+ \frac{C_4}{k^N}\left( \int_{\R^N \setminus B(-y,k\rho)} |\varphi (x)|^{2_{s}^{*}}\right)^{\frac{2}{2_{s}^{*}}}\\
&\leq \frac{\tilde{C}}{\rho^{2s}}\int_{B(-y, k\rho)}|\varphi (x)|^2dx +  \frac{C_4}{k^N}\left( \int_{\R^N \setminus B(-y,k\rho)} |\varphi (x)|^{2_{s}^{*}}\right)^{\frac{2}{2_{s}^{*}}}\\
&\leq \tilde{C}k^{{2s}}\left(\int_{B(-y,k\rho)}|\varphi (x)|^{2_{s}^{*}}dx\right)^{\frac{2}{2_{s}^{*}}} + \frac{\tilde{C}_4}{k^N},
\end{aligned}
$$
for some positive constants $\tilde{C}$ and $\tilde{C}_4$.\par
\noindent Now, given $\varepsilon >0$, we can fix $k$ large enough such that $$\frac{\tilde{C}_4}{k^N}< \frac{\varepsilon}{2}.$$ Hence,
$$
{I_1} \leq \tilde{C}k^{{2s}}\left(\int_{B(-y,k\rho)}|\varphi (x)|^{2_{s}^{*}}dx\right)^{\frac{2}{2_{s}^{*}}} + \frac{\varepsilon}{2}.
$$
Moreover, let us fix $\rho$ small enough such that
$$
\int_{B(-y,k\rho)}|\varphi (x)|^{2_{s}^{*}}dx < \left(\frac{\varepsilon}{2}\right)^{\frac{2^*_s}{2}}\frac{1}{\tilde{C}k^{2s}},
$$
for every $y\in \R^N$.\par
\noindent Hence, ${I_1} < \varepsilon$ uniformly in $y$ for $\rho$ small enough, that is,
$$
\lim_{\rho \to 0}{I_1} =0,
$$
for every $y\in\R^N$. Now, by (\ref{T01}), (\ref{T02-0}) and (\ref{T02}), we get
%\begin{equation}\label{T16}
%\|v_{y}^{\rho} - \varphi (\cdot - y)\|_{s} \to 0 \;\;\mbox{as}\;\;\rho \to \infty\;\;\mbox{uniformly in}\;\;y.
%\end{equation}
%Which implies that
\begin{equation}\label{T18}
\|v_{y}^{\rho}\|_s\to \|\varphi\|_{s} = M\;\;\mbox{as}\;\;\rho \to 0,
\end{equation}
for every $y\in\R^N$.\par
\noindent Therefore, combining (\ref{T01-0}) and (\ref{T18}), we find
$$
\phi_\rho(y) = \frac{v_y^\rho}{\|v_y^\rho\|_{L^p(\R^N)}} \to \varphi (\cdot - y)\;\;\mbox{in}\;\;H^s(\R^N)\;\;\mbox{as}\;\;\rho \to 0,
$$
for every $y\in\R^N$.\par

\bigskip

\noindent
 Part {$(ii)$} - For each $\rho$ fixed, let us consider an arbitrary sequence $\{y_n \}\subset \R^N$ with $|y_n|\to +\infty$ as $n\to +\infty$. As in the proof of Theorem \ref{Etm01} we can show that
$$
\|v_{y_n}^{\rho} - \varphi(\cdot - y_n)\|_{L^p(\R^N)} \to 0\;\;\mbox{and}\;\;\|v_{y_n}^{\rho} - \varphi (\cdot - y_n)\|_s^2 \to 0
$$
as $|y_n| \to +\infty$. On the other hand, since $\{y_n\}$ is arbitrary, it follows that
$$
\|\phi_\rho(y)\|_{s}^{2} = \left\| \frac{v_y^\rho}{\|v_{y}^{\rho}\|_{L^p(\R^N)}}\right\|_s^2 \to \frac{\|\varphi (\cdot -y)\|_{s}^{2}}{\|\varphi(\cdot-y)\|_{L^p(\R^N)}^{2}} = \frac{\|\varphi\|_{s}^{2}}{\|\varphi\|_{L^p(\R^N)}^2} = M
$$
as $|y|\to +\infty$.
\end{proof}

\begin{corollary}\label{Tcor01}
There is $\tilde{\rho}>0 $ such that
\begin{equation}\label{T19}
\sup_{y\in \R^N}\|\phi_\rho(y)\|_{s}^{2} < 2^{\frac{p-2}{p}}M,
\end{equation}
for every $\rho \leq \tilde{\rho}$.
\end{corollary}
\begin{proof}
Since $p>2$, then %$2^{\frac{p}{p-1}}>1$ and then
$M < 2^{\frac{p-2}{p}}M$. By Lemma \ref{Tlm01} - Part $(ii)$, one has
$$
\lim_{\rho \to 0}\|\phi_\rho (y)\|_s^2 = M,
$$
for every $y\in\R^N$. So, given $\varepsilon \in \left(0, 2^{\frac{p-2}{p}}M -M\right)$, there is $\tilde{\rho}>0$ such that
$$
M \leq \|\phi_\rho (y)\|_{s}^{2}\leq M + \epsilon,\;\;\forall \rho \leq \tilde{\rho}
$$
for every $y\in\R^N$.\par
\noindent Consequently
$$
\sup_{y\in \R^N}\|\phi_\rho (y)\|_{s}^{2} \leq M +\varepsilon < M + 2^{\frac{p-2}{p}}M - M < 2^{\frac{p-2}{p}}M,
$$
which proves the claim.
\end{proof}
Hereafter, let us fix $\rho < \tilde{\rho}$, where $\rho$ is the smallest positive number such that
$$
\R^N \setminus \Omega \subset B(0, \rho).
$$
Furthermore, consider the barycenter function given by
$$
\begin{aligned}
\tau : H^s(\R^N)&\to \R^N\\
u&\mapsto \tau (u) := \int_{\R^N} |u(x)|^2\chi (|x|)xdx,
\end{aligned}
$$
where $\chi \in C(\R^+, \R)$ is a non-increasing real function such that
$$
\chi (t) := \begin{cases}
1&t\in (0,R]\\
\displaystyle\frac{R}{t}&t>R,
\end{cases}
$$
for some $R>0$ for which $\R^N \setminus \Omega \subset B(0, R)$. By definition of $\chi$, of course
%$$
%\begin{aligned}
%&\chi (|x|)|x| \leq R,\;\;\mbox{if}\;\;|x|\leq R\;\;\mbox{and}\;\;\chi (|x|)|x| = \frac{R}{|x|}|x| = R,\;\;\mbox{if}\;\;|x|>R.
%\end{aligned}
%$$
%Therefore
\begin{equation}\label{T20-0}
\chi (|x|)|x| \leq R,\;\;\;\;\forall x\in \R^N.
\end{equation}
Set
$$
\mathcal{T}_0 := \{u\in \mathcal{V}\cap \mathcal{P}:\;\;\tau (u) = 0\} \subset X_0^s.
$$
%where $P$ and $V$ are given above.

\begin{lemma}\label{Tlm02}
If $\displaystyle c_0 := \inf_{u\in \mathcal{T}_0}\|u\|^2$, then
\begin{equation}\label{T20}
M < c_0,
\end{equation}
and there is $R_0$, with $R_0>\rho$, such that:
\begin{itemize}
\item[$(i)$] If $y\in\R^N$, with $|y|\geq R_0$, then $$\|\phi_\rho(y)\|_s^2 \in \left(M, \frac{c_0 + M}{2}\right);$$
\item[$(ii)$] If $y\in\R^N$, with $|y| = R_0$, then $$\langle \tau(\phi_\rho (y)), y \rangle>0.$$
\end{itemize}
\end{lemma}
\begin{proof}
Since
$$
c_0 = \inf_{u\in \mathcal{T}_0}\|u\|^2\;\;\mbox{and}\;\;M =\inf_{u\in \mathcal{V}}\|u\|^2,
$$
%Moreover by definition $T_0\subset V$,
we have
$$
M \leq c_0.
$$
Now we are going to show that $c_0\neq M$. Suppose by contradiction that $c_0 = M$. Then, there is a minimizing sequence $\{v_n\}\subset X_0^s$ such that
$$
 \|v_n\|^2 \to M \;\;\mbox{as}\;\;n\to +\infty, \quad \|v_n\|_{L^p(\Omega)} = 1\;\;\mbox{and}\;\;\tau (v_n) = 0,\;\;\forall n\in \N.
$$
By using the Ekeland variational principle, we can suppose that
\begin{equation}\label{T24}
\|J'({v}_n)\|_{*} \to 0.
\end{equation}
By considering the sequence ${u}_n := M^{\frac{1}{p-2}}{v}_n$ it easily seen that
$$
I'({u}_n) = o_n(1) \;\;\mbox{in}\;\;(X_0^s)^*
$$
and
$$
I({u}_n)= \left(\frac{1}{2}-\frac{1}{p} \right)M^{\frac{p}{p-2}} + o_n(1).
$$
Moreover, by Lemma \ref{Elm01}, one has
$$
\|u_n\|^2 \to \|u^0\|^2 +\sum_{j=1}^{k}\|u^j\|_{s}^{2}
$$
and
$$
I(u_n) \to I(u^0) + \sum_{j=1}^{k}I_{\infty}(u^j).
$$
Thereby,
$$
I(u_n) \to I(u^0) + \sum_{j=1}^{k}I_{\infty}(u^j) \geq I(u^0) + k\left( \frac{1}{2}-\frac{1}{p}\right)M^{\frac{p}{p-2}}.
$$
Since $I(u^0)\geq 0$, then $k\leq 1$, bearing in mind that
$$
I(u_n) = \left(\frac{1}{2} - \frac{1}{p} \right)M^{\frac{p}{p-2}} + o_n(1).
$$
Hence, we must have either $k=0$ or $k=1$. If $k = 0$, we obtain that $\|u_n\|^2 \to \|u^0\|^2$, which leads to
$$
u_n \to u^0\;\;\mbox{in}\;\;X_0^s \quad \mbox{and} \quad \|u^0\|^2 = M^{\frac{p}{p-2}}.
$$
This is impossible, because $M$ is not achieved in $\mathcal{V}$, and so, $k\neq 0$. For $k=1$, we must have $u^0=0$. Consequently,
$$
\|u_n\|^2 \to \|u^1\|_s^2\;\;\mbox{and}\;\;I(u^1) = \left(\frac{1}{2}-\frac{1}{p}\right)M^{\frac{p}{p-2}}.
$$
Here we used the fact that $u^1$ must be a positive ground state solution of
$$
\left\{
\begin{aligned}
(-\Delta)^s u+ u &= |u|^{p-2}u,\;\;\mbox{in}\;\;\R^N\\
u&\in H^s(\R^N),
\end{aligned}
\right.
$$
and so, by uniqueness, $I(u^1) = \left(\frac{1}{2}-\frac{1}{p}\right)M^{\frac{p}{p-2}}$. Since $u_n \rightharpoonup u^0 \equiv 0$, we get
$$
\psi_{n}^{1}(x+y_n^1) = u_n(x+y_n^1) \rightharpoonup u^1(x)
$$
and
$$
\|\psi_{n}^{1}(\cdot + y_n^1)\|_s^2 = \|u_n(\cdot + y_n^1)\|_s^2 = \|u_n\|_{s}^{2}= \|u_n\|^2 \to \|u^1\|_{s}^{2},
$$
where $\{y_n^1\}$ be a sequence such that $|y_n^1| \to +\infty$. Therefore
$$
u_n(\cdot + y_n^1) \to u^1\;\;\mbox{in}\;\;H^s(\R^N).
$$
Taking
$$
u^1 = u,\;\;y_n^1 = y_n\;\;\mbox{and}\;\;w_n(x + y_n^1) = u_n(x + y_n^1) - u^1(x),
$$
we have
$$
w_n(x) = u_n(x) - u(x-y_n),\;\;\; \forall x\in \R^N,
$$
and
$$
\|w_n\|_{s}^{2} = \|w_n(\cdot + y_n)\|_{s}^{2} = \|u_n(\cdot + y_n) - u\|_{s}^{2}.
$$
Therefore, the strong convergent of $u_n(\cdot +y_n)$ yields
$$
w_n\to 0\;\;\mbox{in}\;\;H^s(\R^N).
$$
Next, we consider the following sets
$$
(\R^N)_{n}^{+} = \{x\in \R^N:\;\; \langle x , y_n \rangle > 0\}\;\;\mbox{and}\;\;(\R^N)_n^- = \R^N \setminus (\R^N)_{n}^{+}.
$$
Using the fact that $|y_n|\to +\infty$ as $n\to +\infty$, we claim that there is a ball
$$
B(y_n, \tilde{r}) = \{x\in \R^N: \;\;|x-y_n|< \tilde{r}\}\subset (\R^N)_{n}^{+}
$$
such that
\begin{equation}\label{T28}
u(x-y_n) \geq \frac{1}{2}u(0)>0, \;\;\;\; \forall x\in B(y_n, \tilde{r}),
\end{equation}
for $n$ large enough. %Moreover
%$$
%{\red u(x-y_n) \leq \frac{K}{|x-y_n|}\;\;\forall x\in (\R^N)_{n}^{-}}
%$$
Indeed, firstly we recall that $u(0)$ is the maximum value of $u$ in $\R^N$. As $u$ be a positive radial decreasing function, then
$$
u(z) \leq \left(\frac{N}{|S^{N-1}|} \right)^{1/2} \frac{\|u\|_{L^2(\R^N)}}{|z|^{N/2}},\;\;\forall z\neq 0, \quad ( \mbox{see \cite{HBPL}} )
$$
which implies that
$$
u(z)\to 0\;\;\mbox{as}\;\;|z|\to +\infty.
$$
Then by the Intermediate value theorem, there exists $\tilde{r}>0$ such that
\begin{equation}\label{T27}
u(z) = \frac{1}{2}u(0)>0,\;\;\forall z\in \R^N\;\;\mbox{with}\;\;|z| = \tilde{r}.
\end{equation}
Substituting $z = x-y_n$ into (\ref{T27}), we get (\ref{T28}). On the other hand, for each $\tilde{r}>0$ fixed, %since $|y_n|\to +\infty$,
there is $n_0$ such that
$$
\begin{aligned}
\langle x , y_n \rangle  %& = \frac{|x|^2 + |y_n|^2 - |x-y_n|^2}{2}\\
&> \frac{|x|^2 + |y_n|^2 - \tilde{r}^2}{2} \geq \frac{|y_n|^2 - \tilde{r}^2}{2} >0, \;\;\forall n\geq n_0, \quad \forall x\in B(y_n, \tilde{r}),
\end{aligned}
$$
showing that
$$
B(y_n, \tilde{r}) \subset (\R^N)_{n}^{+},\;\;\mbox{for $n$ large enough}.
$$
%then $|x-y_n| > 1$ for $n$ large enough, then by theorem 1.5 in \cite{PFAQJT}, there exist $C_2\geq C_1 >0$ such that
%\begin{equation}\label{T29}
%\frac{C_1}{|x-y_n|^{N+2s}}\leq u(x-y_n)\leq \frac{C_2}{|x-y_n|^{N+2s}}.
%\end{equation}
%Therefore
%$$
%\begin{aligned}
%\tau(u(x-y_n))\cdot y_n & = \int_{\R^N} |u(x-y_n)|^2\chi (\|x\|) x\cdot y_n dx\\
%& = \int_{(\R^N)_{n}^{+}} |u(x-y_n)|^2\chi (\|x\|) x\cdot y_n dx + \int_{(\R^N)_n^-} |u(x-y_n)|^2\chi (\|x\|) x\cdot y_n dx.
%\end{aligned}
%$$
Thus, for $n$ large enough,
$$
|u(x-y_n)|^2, \chi (|x|), \langle x\ , y_n \rangle  >0,\;\;\forall x\in (\R^N)_{n}^{+}, \quad B(y_n, \tilde{r}) \subset (\R^N)_{n}^{+}
$$
and $|x|>R$ for every $x\in B(y_n, \tilde{r})$. Using these informations, a straightforward computation gives
\begin{equation}\label{T29}
\begin{aligned}
\int_{(\R^N)_{n}^{+}} |u(x-y_n)|^2\chi (|x|) \langle x , y_n \rangle dx %&\geq \int_{B(y_n, \tilde{r})} |u(x-y_n)|^2\chi (|x|) x\cdot y_n dx\\
%& \geq \int_{B(y_n, \tilde{r})} \frac{|u(0)|^2}{2} \chi (|x|)|x||y_n|\cos \theta dx\\
&\geq \frac{R|u(0)|^2}{4}med(B(y_n, \tilde{r}))|y_n|.
\end{aligned}
\end{equation}

Recalling that for each $x\in (\R^N)_{n}^{-}$,
$$
|x-y_n|\geq |x|,
$$
it follows that
$$
|u(x-y_n)|^2\chi(|x|)|x| \leq R |u(|x|)|^2 \in L^1(\R^N).
$$
This combined with the limit
$$
u(x-y_n) \to 0\;\;\mbox{as}\;\;|y_n|\to +\infty
$$
implies that
\begin{equation}\label{T31}
\int_{(\R^N)_{n}^{-}} |u(x-y_n)|^2\chi (|x|)|x|dx = o_n(1).
\end{equation}
Therefore, by the Cauchy-Schwartz inequality and (\ref{T31}),
\begin{equation}\label{T32}
\begin{aligned}
\left\langle \tau(u(x-y_n)),  \frac{y_n}{|y_n|}\right\rangle_{\R^N} & =  \int_{(\R^N)_{n}^{+}} |u(x-y_n)|^2\chi (|x|) \langle x , y_n \rangle dx \\
&\,\,\,\,\,\,+ \int_{(\R^N)_n^-} |u(x-y_n)|^2\chi (|x|) x\cdot y_n dx\\
&\geq \frac{R|u(0)|^2}{4}med(B(y_n, \tilde{r})) - \int_{(\R)_{n}^{-}} |u(x-y_n)|^2\chi (|x|)|x|dx\\
&\geq \frac{R|u(0)|^2}{4}med(B(y_n, \tilde{r})) -o_n(1)>0.
\end{aligned}
\end{equation}
Now, using the fact that $w_n\to 0$ in $H^s(\R^N)$ together with $\tau(u_n) = 0$, we find that
\begin{equation}\label{T33}
\tau (u(x-y_n)) = o_n(1),
\end{equation}
which contradicts (\ref{T32}), and so, $M < c_0$.

\smallskip
\noindent
\indent Proof of Part {$(i)$} - Since $\phi_\rho(y) \in X_0^s$ and $\|\phi_\rho (y)\|_{L^p} = 1$, by Theorem \ref{Etm01} we must have
\begin{equation}\label{T34}
\|\phi_\rho(y)\|^2 > M,\,\;\forall y\in \R^N.
\end{equation}
By Lemma \ref{Tlm01} - (ii), for each $\rho$ fixed
\begin{equation}\label{T35}
\|\phi_\rho(y)\|^2 \to M,\;\,\mbox{as}\;\;|y|\to +\infty.
\end{equation}
Thereby, for a given $\epsilon \in (0, \frac{c_0-M}{2}),$ there is $R_0>0$ such that
\begin{equation*}\label{T36}
|\|\phi_\rho (y)\|^2 - M| < \epsilon \;\;\mbox{whenever}\;\;|y|\geq R_0.
\end{equation*}
From this,
$$
\|\phi_\rho (y)\|^2\in \left(M, \frac{c_0+M}{2} \right),\;\;\forall y\in \R^N\;\;\mbox{such that}\;\;|y|\geq R_0.
$$

\noindent
\indent Proof of Part {$(ii)$} - By definition of $\phi_\rho (y)$ and arguing as above with  $|y|$ large enough, we have
$$
\begin{aligned}
\left\langle \tau (\phi_\rho (y)), y\right\rangle_{\R^N} %&= \int_{\R^N} |\phi_\rho (y)|^2\chi(|x|)x\cdot y dx\\
%&= \frac{1}{\|v_y^{\rho}\|_{L^p(\R^N)}^2}\int_{\R^N}\left|\xi \left(\frac{|x|}{\rho}\right)\varphi(x-y) \right|^2\chi(|x|)x\cdot ydx\\
&= \frac{1}{\|v_y^{\rho}\|_{L^p(\R^N)}^2}\int_{(\R^N)^+}\left|\xi \left(\frac{|x|}{\rho}\right)\varphi(x-y) \right|^2\chi(|x|)\langle x , y \rangle dx\\
&+ \frac{1}{\|v_y^{\rho}\|_{L^p(\R^N)}^2}\int_{(\R^N)^-}\left|\xi \left(\frac{|x|}{\rho}\right)\varphi(x-y) \right|^2\chi(|x|)\langle x , y \rangle dx\\
&\geq \frac{1}{\|v_{y}^{\rho}\|_{L^p(\R^N)}} \int_{B(y, \tilde{r})}\left|\xi\left(\frac{|x|}{\rho}\right) \right|^2\frac{|\varphi(0)|^2}{4}\chi (|x|)\langle x , y \rangle dx dx -o(1).
\end{aligned}
$$
Hence, for $|y| = R_0$,
\begin{equation}\label{T38}
\left\langle \tau (\phi_\rho (y)), y\right\rangle_{\R^N} \geq \frac{RR_0}{4}|\varphi (0)|^2medB(B(y, \tilde{r})) - o(1) >0.
\end{equation}
\end{proof}

\section{Proof of Theorem \ref{teo 1}}

\noindent
From now on, we set $\Sigma \subset \mathcal{P} \subset X_0^s$ defined as follows
$$
\Sigma := \{\phi_\rho (y):\;|y|\leq R_0\},
$$
$$
\mathcal{H} := \left\{h\in C(\mathcal{P}\cap \mathcal{V}, \mathcal{P}\cap \mathcal{V}):\; h(u) = u,\;\forall u\in \mathcal{P}\cap \mathcal{V}\;\;\mbox{such that}\;\;\|u\|^2 < \frac{c_0 + M}{2} \right\}
$$
and
$$
\Gamma := \{A\subset \mathcal{P}\cap \mathcal{V}:\;A= h(\Sigma),\;h\in \mathcal{H}\}.
$$
\begin{lemma}\label{Tlm03}
If $A\in \Gamma$, then $A \cap \mathcal{T}_0 \neq \emptyset$.
\end{lemma}
\begin{proof}
We are going to show that, for every $A\in \Gamma$, there exists $u\in A$ such that $\tau(u) = 0$. Equivalently, we prove that: for every $h\in \mathcal{H}$, there exists $\tilde{y}\in \R^N$ with $|\tilde{y}|\leq R_0$ such that
\begin{equation}\label{T39}
(\tau \circ h \circ \phi_\rho)(y) = 0.
\end{equation}
For any $h\in \mathcal{H}$, we can define the function
$$
\mathcal{J}:=\tau \circ h \circ \phi_\rho: \R^N \to \R^N
$$
and $\mathcal{F}:[0,1]\times \overline{B}_{R_0}(0)\rightarrow \R^N$ given by
$$
\mathcal{F} (t, y) := t\mathcal{J}(y) + (1-t)y.
$$
We claim that $0\not \in \mathcal{F}(t, \partial B(0, R_0))$. Indeed, for $|y| = R_0$, by Lemma \ref{Tlm02} - Part $(i)$ we have
$$
\|\phi_\rho (y)\|^2 < \frac{c_0+M}{2}.
$$
Hence, it follows that
$$
\mathcal{F}(t,y)=t(\tau \circ \phi_\rho)(y) + (1-t)y,
$$
and
\begin{equation}\label{T40}
\langle \mathcal{F}(t,y), y\rangle = t \langle \tau(\phi_\rho (y)), y\rangle + (1-t)\langle y,y\rangle.
\end{equation}
Now
\begin{enumerate}
\item[$\circ$] If $t=0$, then $\langle \mathcal{F}(0,y), y\rangle = |y|^2 = R_{0}^{2}>0$;
\item[$\circ$] If $t=1$, then by Lemma \ref{Tlm02} - Part $(ii)$ we have $\langle \mathcal{F}(1,y), y\rangle = \langle \tau(\phi_\rho(y)), y \rangle >0$;
\item[$\circ$] If $t\in (0,1)$, then $\langle \mathcal{F}(t,y), y\rangle >0$, since the terms $t, 1-t, \langle \tau(\phi_\rho (y)), y\rangle $ and $|y|^2$ are positives.
\end{enumerate}
Then, by using the invariance under homotopy of the Brouwer degree, one has
$$
d(\mathcal{F}(t, \cdot), B(0, R_0), 0) = \mbox{constant},\;\;\forall t\in [0,1].
$$
%Note that, for $t=0$
%$$
%d(\mathcal{F}(0,\cdot), B(0, R_0), 0) = d(I, B(0, R_0), 0) = 1.
%$$
%On the other hand
%$$
%d(\mathcal{F}(1, \cdot), B(0, R_0), 0) = d(\mathcal{J}, B(0, R_0),0)
%$$
%So
Since
$$
d(\mathcal{J}, B(0, R_0), 0) = 1\neq 0,
$$
there exists $\tilde{y}\in B(0, R_0)$ such that $\mathcal{J}(\tilde{y}) = 0$, that is
$$
\mathcal{J}(\tilde{y}) = (\tau \circ h \circ \phi_\rho)(\tilde{y}) = 0.
$$
This completes the proof.
\end{proof}

\noindent
Now, let us denote
\begin{equation}\label{T41}
c := \inf_{A\in \Gamma} \sup_{u\in A} \|u\|^2,
\end{equation}
$$
\mathcal{K}_c := \{u\in \mathcal{P}\cap \mathcal{V}:\; J(u) = \|u\|^2 = c\;\;\mbox{and}\;\;\nabla J|_{V}(u) = 0\},
$$
and set
$$
L_\gamma := \{u\in \mathcal{V}:\;J(u)\leq \gamma\},
$$
for every $\gamma \in \R$.\par
\smallskip
\noindent {\bf Proof of Theorem \ref{teo 1}}
\begin{proof}
We choose $\rho = \tilde{\rho}$ that is given in Corollary \ref{Tcor01}. We claim that $c$ given by (\ref{T41}) is a critical value, that is, $\mathcal{K}_c \neq \emptyset$. We start our analysis by noting that
\begin{equation}\label{T42}
M < c < 2^{\frac{p-2}{p}}M.
\end{equation}
In fact, by Lemma \ref{Tlm03}, for every $A\in \Gamma$, $A\cap \mathcal{T}_0 \neq \emptyset$. Then, for each $A\in \Gamma$ there is $\tilde{u}\in A \cap \mathcal{T}_0$ such that
\begin{equation}\label{T43}
\inf_{u\in \mathcal{T}_0}\|u\|^2 \leq \inf_{u\in A\cap \mathcal{T}_0} \|u\|^2 \leq \|\tilde{u}\|^2 \leq \sup_{u\in A\cap \mathcal{T}_0}\|u\|^2\leq \sup_{u\in A}\|u\|^2.
\end{equation}
Moreover, by Lemma \ref{Tlm02} and (\ref{T43}), we obtain
$$
M < c_0 = \inf_{u\in \mathcal{T}_0}\|u\|^2 \leq \sup_{u\in A}\|u\|^2, \quad \forall A\in \Gamma.
$$
Thus
\begin{equation}\label{T44}
M < c_0\leq \inf_{A\in \Gamma} \sup_{u\in A}\|u\|^2 = c.
\end{equation}
Owing to
\begin{equation}\label{T45}
c\leq \sup_{u\in A}\|u\|^2, \;\;\forall A\in \Gamma,
\end{equation}
it follows that
$$
c\leq \sup_{\phi_\rho (y) \in \Sigma} \|h(\phi_\rho (y))\|^2,\;\;\forall h\in \mathcal{H}.
$$
Now taking $h \equiv I$, we find
$$
c\leq \sup_{\phi_\rho (y)\in \Sigma}\|\phi_\rho (y)\|^2.
$$
Thus,
$$
c\leq \sup_{|y|\leq R_0} \|\phi_\rho (y)\|^2 \leq \sup_{y\in \R^N}\|\phi_\rho (y)\|^2,
$$
and by Corollary \ref{Tcor01},
$$
c\leq \sup_{y\in \R^N}\|\phi_\rho(y)\|^2 < 2^{\frac{p-2}{2}}M.
$$
The last inequality, in addition to (\ref{T44}), yields
\begin{equation}\label{T46}
M < c < 2^{\frac{p-2}{2}}M.
\end{equation}
Next, by Corollary \ref{Ecor03}, the functional $J$ satisfies the Palais-Smale condition at the level $c$ in the following set
$$
Z := \mathcal{P}\cap \mathcal{V} \cap \{u\in X_0^s:\;M < J(u) < 2^{\frac{p-2}{2}}M\}.
$$
\indent
Suppose by contradiction that $\mathcal{K}_c = \emptyset$. The following inequality
$$
\frac{c_0+M_\lambda}{2} \leq \frac{c+M}{2}< \frac{c+c}{2} = c < 2^{\frac{p-2}{p}}M,
$$
in addition to the Deformation Lemma guarantees the existence of a continuous map
$$
\eta:[0,1]\times \mathcal{V}\cap \mathcal{P}\to \mathcal{V}\cap \mathcal{P}
$$
and a positive number $\varepsilon_0$ such that
\begin{itemize}
\item[$(a)$] $L_{c+\varepsilon_0}\setminus L_{c-\varepsilon_0} \subset\subset L_{2^{\frac{p-2}{p}}M} \setminus L_{\frac{c_0+M}{2}}$,
\item[$(b)$] $\eta (t,u) = u,\;\;\forall u\in L_{c-\varepsilon_0} \cup \{\mathcal{V}\cap \mathcal{P} \setminus L_{c+\varepsilon_0}\}\;\;\mbox{and}\;\;\forall t\in [0,1]$,
\item[$(c)$] $\eta (1, L_{c + \frac{\varepsilon_0}{2}}) \subset L_{c - \frac{\varepsilon_0}{2}}$.
\end{itemize}

Fix $\tilde{A}\in \Gamma$ such that
$$
c\leq \sup_{u\in \tilde{A}} J(u) < c + \frac{\varepsilon_0}{2}.
$$
Since
$$
J(u) < c + \frac{\varepsilon_0}{2},\;\;\forall u\in \tilde{A},
$$
it follows that
$$
\tilde{A} \subset L_{c+\frac{\varepsilon_0}{2}}.
$$
Now, by using again the Deformation Lemma, one has
%$$
%\eta (1, \tilde{A}) \subset L_{c-\frac{\varepsilon_0}{2}},
%$$
%namely
$$
J(u) < c-\frac{\varepsilon_0}{2},\;\;\forall u\in \eta (1, \tilde{A}),
$$
that is,
\begin{equation}\label{T47}
\sup_{u\in \eta(1, \tilde{A})}J(u) < c-\frac{\varepsilon_0}{2}.
\end{equation}
On the other hand, we notice that $\eta (1, \cdot) \in C(\mathcal{V}\cap \mathcal{P}, \mathcal{V}\cap \mathcal{P})$. Moreover, since $\tilde{A} \in \Gamma$, there exists $h\in \mathcal{H}$ such that $\tilde{A} = h(\Sigma)$. Consequently,
$$
\tilde{h} = \eta(1, \cdot) \circ h\in C(\mathcal{V}\cap \mathcal{P}, \mathcal{V}\cap \mathcal{P}).
$$
Since $h\in \mathcal{H}$, it follows that
$$
h(u) = u,\;\;\forall u\;\;\mbox{such that}\;\;\|u\|^2 < \frac{c_0 + M}{2} < c-\varepsilon_0,
$$
and
$$
\tilde{h}(u) = \eta (1,u)\;\;\forall u\;\;\mbox{such that}\;\;\|u\|^2 < \frac{c_0+M}{2}.
$$
Taking into account that
$$
\frac{c_0+M_\lambda}{2} < c-\epsilon_0,
$$
by item $(b)$, we easily have
$$
\tilde{h}(u) = \eta(1,u) = u,\;\;\forall u \;\;\mbox{such that}\;\;\|u\|^2 < \frac{c_0+M}{2} < c - \epsilon_0.
$$
Then $\tilde{h} \in \mathcal{H}$. Moreover
$$
\eta(1,\tilde{A}) \in \Gamma,
$$
owing to $\eta(1, \tilde{A}) = \tilde{h}(\Sigma)$.
Therefore, exploiting the definition of $c$, we have
$$
c\leq \sup_{u\in \eta (1, \tilde{A})} J(u),
$$
which contradicts (\ref{T47}). Thereby, $\mathcal{K}_c \neq \emptyset$ and $c$ is a critical value of functional $J$ on  $\mathcal{V} \cap \mathcal{P}$, namely there is at least one nonnegative solution of ($P$).
\end{proof}


\begin{thebibliography}{1}
	

\bibitem{CALF} C.O. Alves and L. Freitas, {\it Existence of a Positive Solution for a Class of Elliptic Problems in Exterior Domains Involving Critical Growth}, Milan J. Math. {\bf 85}, 309-330 (2017).

\bibitem{AlvesMIyagaki1} C.O. Alves and O.H. Miyagaki, {\it Existence and concentration of solution for a class of fractional  elliptic equation in $\R^N$ via penalization method}, Calc. Var. Partial Differential Equations, \textbf{55} (2016), n 3, Art. 47, 19 pp.

\bibitem{Alves-delima-Nobrega} C. O. Alves, R. N. de Lima and A. B. N\'obrega, {\it Bifurcation properties for a class of fractional Laplacian equations in $\R^N$}, to appear in Math. Nachr. (2018).

\bibitem{Autuori}G. Autuori, P. Pucci, {\it Elliptic problems involving the fractional Laplacian in $\mathbb{R}^N$}, J. Differential Equations \textbf{255}, 2340--2362 (2013).

\bibitem{bah} A. Bahri and P. L. Lions {\it On the existence of a positive solution of semilinear elliptic equations in unbounded domains}, Ann. Inst. H. Poincaré
Anal. Non Lin\'eaire 14 (1997), 365-413.

\bibitem{VBGC}V. Benci and G. Cerami, {\it Positive solutions of some nonlinear elliptic problems in exterior domains}, Arch. Rational Mech. Anal. {\bf 99}, 283-300 (1987).

\bibitem{HBPL}H. Berestycki and P.-L. Lions, {\it Nonlinear scalar field equations. I. Existence of a ground state}, Arch. Rational Mech. Anal. {\bf 82}, 4, 313-345 (1983).

\bibitem{Bucurb}C. Bucur and E. Valdinoci, {\it Nonlocal Diffusion and Applications}, Springer International Publishing Switzerland 2016.

\bibitem{Caponi}M. Caponi, P. Pucci, {\it Existence theorems for entire solutions of stationary Kirchhoff fractional $p$--Laplacian equations}, Ann. Mat. Pura Appl. {\bf 195}, 2099-2129 (2016).

\bibitem{BP1} G. Cerami and D. Passaseo, { \it Existence and multiplicity results for semilinear elliptic Dirichlet problems in exterior domains}, Nonlinear Anal., T.M.A. 24(11) (1995) 1533-1547.

\bibitem{MC}M. Cheng, {\it Bound state for the fractional Schr\"odinger equation with unbounded potential.} J. Math. Phys. {\bf 53}, 043507 (2012).


\bibitem{citti} G. Citti, {\it On the exterior Dirichlet problem for $\Delta u -u + f(x,u) =0$},
Rend. Sem. Mat. Univ. Padova 88 (1992), 83-110.

\bibitem{clapp} M. Clapp and  D. Salazar, {\it Multiple sign changing solutions of nonlinear elliptic problems in exterior domains},  Adv. Nonlinear Stud. 12 no. 3 (2012), 427-443.


\bibitem{coffman} C.V. Coffman and M. Marcus, { \it Existence theorems for superlinear elliptic
	Dirichlet problems in exterior domains}, Nonlinear Functional Analysis, and Its Applications, Proceedings of Symposia in Pure Mathematics, Providence 1986, Vol. 45, Amer. Math. Soc.


\bibitem{Davila} J. D\'avila, M. del Pino and J. C. Wei, {\it  Concentrating standing waves for the fractional nonlinear Schr\"odinger equation,} J. Differential Equations {\bf 256} (2014), 858-892.


%\bibitem{esteban} M. J. Esteban and P. L. Lions, { \it Existence and non-existence results for semilinear
%	elliptic problems in unbounded domains}, Proc. Royal Edinbourgh Soc. 93 A (1982), 1-14.


\bibitem{EDNGPEV}E. Di Nezza, G. Palatucci and  E. Valdinoci, {\it Hitchhiker's guide to the fractional Sobolev spaces}. Bull. Sci. math. {\bf 136}, 521-573 (2012).

\bibitem{Dipierrob}S. Dipierro, M. Medina and E. Valdinoci, {\it Fractional Elliptic Problems
with Critical Growth in the whole of $\R^N$}, Lecture Notes. Scuola Normale Superiore di Pisa (New Series), 15. Edizioni della Normale, Pisa, 2017.

\bibitem{SDGPEV}S. Dipierro, G. Palatucci and E. Valdinoci, {\it Existence and symmetry results for a Schrödinger type problem involving the fractional Laplacian}, Le Matematiche, {\bf LXIII}, I, 201-216 (2013).

\bibitem{Moustapha} M. M. Fall, F. Mahmoudi and E. Valdinoci, {\it Ground states and concentration phenomena for the fractional Schr\"odinger equation,}
Nonlinearity \textbf{28}(6), (2015) 1937-1961.

\bibitem{PFAQJT}P. Felmer, A. Quaas and J. Tan, {\it Positive solutions of nonlinear Schr\"odinger equation with the fractional laplacian}, Proceedings of the Royal Society of Edinburgh: Section A Mathematics, {\bf 142}, No 6, 1237-1262 (2012).

\bibitem{Fiscella}A. Fiscella, P. Pucci and S. Saldi, {\it Existence of entire solutions for Schr\"odinger--Hardy systems involving two fractional operators}, {Nonlinear Anal.} {\bf158}, 109-131 (2017).

\bibitem{RFEL}R. Frank and E. Lenzmann, {\it Uniqueness of nonlinear ground states for fractional Laplacians in $\R$}, Acta Math. {\bf 210}, 2, 261-318 (2013).

\bibitem{RFELLS}R. Frank, E. Lenzmann and L. Silvestre, {\it Uniqueness of radial solutions for the fractional Laplacian}, Comm. Pure Appl. Math. {\bf 69}, 1671-1725 (2016).


\bibitem{li} G. Li and G. F. Zheng, {\it The existence of positive solution to some asymptotically linear elliptic equations in exterior domains}, Rev. Mat. Iberoamericana 22 (2006), no. 2, 559-590.

\bibitem{maia} L. A. Maia and B. Pellacci, {\it Positive solutions for asymptotically linear problems in exterior domains}, Annali di Matematica (2016), 1-32.

\bibitem{MOLRAD} G. Molica Bisci and V. R\u{a}dulescu, {\it Ground state solutions of scalar field fractional Schr\"odinger equations,} Calc.
Var. Partial Differential Equations 54 (2015), no. 3, 2985-3008.

\bibitem{Molicab}G. Molica Bisci, V. R\u{a}dulescu and R. Servadei, {\it Variational Methods for Nonlocal Fractional Problems}, University Printing House, Cambridge CB2 8BS, United Kingdom 2016.

\bibitem{CP}C. Pozrikidis, {\it The Fractional Laplacian}, Taylor \& Francis Group, LLC 2016.

\bibitem{SS-1}S. Secchi, {\it Ground state solutions for nonlinear fractional Schr\"odinger equation in $\mathbb{R}^{N}$}, J. Math. Phys., {\bf 54}, 031501 (2013).

\bibitem{RSEV} R. Servadei and E. Valdinoci, {\it Mountain Pass solutions for non-local elliptic operators}, J. Math. Anal. Appl., {\bf 389}, 887-898 (2012).

\bibitem{RSEV1} R. Servadei and E. Valdinoci, {\it Variational methods for non-local operators of elliptic type}, Discrete Contin. Dyn. Syst., {\bf 33}, 2105-2137 (2013).

\bibitem{ShangZhang} X. Shang and J. Zhang, {\it Concentrating solutions of nonlinear fractional Schr\"odinger equation with potentials,} J. Differential Equations {\bf 258} (2015), 1106--1128.

\bibitem{Shang} X. Shang, J. Zhang and  Y. Yang, \textit{On fractional Schr\"odinger equation in $\R^N$ with critical growth,}
J. Math. Phys. {\bf 54} (2013), 121502-19 pages.


\bibitem{MW}M. Willem, {\it Minimax Theorems}, Birkh\"auser, 1996.

\bibitem{MXBZXZ}M. Xiang, B. Zhang and X. Zhang, {\it A Nonhomogeneous Fractional $p$-Kirchhoff type problem involving critical exponent in $\R^N$}, Adv. Nonlinear Stud. {\bf 17}, 611-640 (2017).

\end{thebibliography}
\end{document}